\documentclass[letterpaper,11pt]{article}

\input{preamble.sty}

\usepackage[suppress]{color-edits}
\addauthor{ja}{BrickRed}
\addauthor{sb}{Green}
\addauthor{ob}{MidnightBlue}

\addauthor{ja}{red}
\addauthor{sb}{magenta}
\addauthor{ob}{blue}

\newcommand{\regret}{\textnormal{Regret}}
\newcommand{\revenue}{\textnormal{Revenue}}
\newcommand{\ratio}{\textnormal{Ratio}}
\newcommand{\all}{\textnormal{All}}

\newcommand{\worstregret}{\textnormal{WorstRegret}}
\newcommand{\worstrevenue}{\textnormal{WorstRevenue}}
\newcommand{\worstratio}{\textnormal{WorstRatio}}

\newcommand{\relperf}{\textnormal{RelPerf}}

\newcommand{\mechregret}{\Phi^*_\regret}
\newcommand{\mechrevenue}{\Phi^*_\revenue}
\newcommand{\mechratio}{\Phi^*_\ratio}
\newcommand{\mechall}{\Phi^*_\textnormal{All}}

\newcommand{\new}{\textnormal{new}}
\newcommand{\old}{\textnormal{old}}

\newcommand{\calFa}{\mathcal{F}^{\textnormal{LB}}_a}
\newcommand{\calFmu}{\mathcal{F}^{\textnormal{mean}}_{\mu}}
\newcommand{\calFmusigma}{\mathcal{F}^{\textnormal{mean+var}}_{\mu,\sigma}}
\newcommand{\calFmedian}{\mathcal{F}^{\textnormal{median}}_{\nu}}

\newcommand{\gridparam}{\mathbb{G}}
\newcommand{\gridval}{\mathcal{G}}

\begin{document}

 \doparttoc % Tell to minitoc to generate a toc for the parts
\faketableofcontents % Run a fake tableofcontents command for the partocs

	\algrenewcommand\algorithmicrequire{\textbf{Input:}}
	\algrenewcommand\algorithmicensure{\textbf{Output:}}
	% ALTERNATE TITLE
	% Robust Mechanism Design with Scale Information
	\title{
The Best of Many Robustness Criteria in  Decision Making:\\
  Formulation and Application to Robust Pricing}

	\ifx\blind\undefined
	\author{ 
	Jerry Anunrojwong\thanks{Columbia University, Graduate School of Business. Email: {\tt janunrojwong25@gsb.columbia.edu}} \and Santiago R. Balseiro\thanks{Columbia University, Graduate School of Business. Email: {\tt srb2155@columbia.edu}} \and Omar Besbes\thanks{ Columbia University, Graduate School of Business. Email: {\tt ob2105@columbia.edu}}
	}
	\fi

% \author{(Authors’ names blinded for peer review)}

\date{\today}

\maketitle
\begin{abstract}

In robust decision-making under non-Bayesian uncertainty, different robust optimization criteria, such as maximin performance, minimax regret, and maximin ratio, have been proposed. In many problems, all three criteria are well-motivated and well-grounded from a decision-theoretic perspective, yet different criteria give different prescriptions. This paper initiates a systematic study of overfitting to robustness criteria. How good is a prescription derived from one criterion when evaluated against another criterion? Does there exist a prescription that performs well against all criteria of interest? We formalize and study these questions through the prototypical problem of robust pricing under various information structures, including support, moments, and percentiles of the distribution of values.  We provide a unified analysis of three focal robust criteria across various information structures and evaluate the relative performance of mechanisms optimized for each criterion against the others. We find that mechanisms optimized for one criterion often perform poorly against other criteria, highlighting the risk of overfitting to a particular robustness criterion. Remarkably, we show it is possible to design mechanisms that achieve good performance across all three criteria simultaneously, suggesting that decision-makers need not compromise among criteria. 

\end{abstract}

\newpage

\setstretch{1.5}

\section{Introduction}

Decisions often need to be made in the face of uncertain environments. If the uncertainty is stochastic and a description of the underlying stochastic elements is known, then this leads to a \emph{stochastic optimization} problem. However, decision makers rarely have perfect information about the environment. More commonly, decision makers know something but not everything, and one wants to incorporate this partial information (through an uncertainty set) in the optimization problem to be ``robust'' to the  possible data generation processes.  

Robust decision-making has a long history in the fields of Statistics, Operations Research, Computer Science and Economics. Significant efforts have been devoted to develop axiom-based criteria \citep{Savage51, Savage54, BorodinElYaniv}, to develop tractable approaches to solve robust decision-making problems \citep{BertsimasDenHertog}, or to understand general questions pertaining to the robust value of partial information. Such kind of robust decision making under non-Bayesian uncertainty requires one to specify the optimality criterion, i.e., to define what one precisely  means by ``robustly optimal.''   There has been significant effort devoted to uncertainty descriptions and understanding which ones lead to tractable problems, and the focus on the criterion of the decision-maker and the definition of robust-optimality has been mainly anchored around axiomatic discussions. 

There are many such criteria proposed in the decision theory literature. One criterion is to evaluate the objective by itself in the worst case. If the non-robust formulation maximizes expected performance, then the \textit{maximin} criterion chooses a decision to maximize worst-case performance. This criterion is sometimes referred to as the Wald criterion after \citet{Wald45}.\footnote{The common formulation is minimax risk; here, given the focal problem we focus on later in which one maximizes performance, we equivalently formulate it as maximin performance.} This worst case formulation is also very common in the robust optimization literature in Operations Research; see \citet{BertsimasDenHertog} for a review. However, a long line of work starting from \citet{Savage51} suggests that the robust performance should be compared against a benchmark of the optimal decision in hindsight; the worst-case gap between the two should be as small as possible.  Benchmark formulations are common in the online algorithms, algorithmic game theory, and machine learning literatures. This comparison/gap can be additive or multiplicative. The additive comparison is the \textit{minimax regret} criterion, minimizing worst-case regret, defined as the difference between the benchmark and the achieved performance \citep{Savage54}. The multiplicative comparison is the \textit{maximin ratio} criterion, maximizing the worst-case ratio between the performance and the benchmark \citep{Karp92}. \citet[Chapter 15]{BorodinElYaniv}  gives an axiomatic treatment of these criteria; different subsets of natural (but mutually incompatible) axioms can give rise to the  maximin performance, minimax regret, and maximin ratio criteria. 

While the three criteria above (maximin performance, minimax regret and maximin ratio) are commonly adopted and studied in the literature, 
 it is natural for decision-makers, when moving from theory to practice,  to ask which criterion they should select. Since the proposal of these criteria, various settings have been exhibited where one criterion may be trivially inappropriate \citep[Chapter 15]{BorodinElYaniv}. As a result, in some problems some criteria may not pass a sanity check and could be safely eliminated.  However, in many decision-making problems, it is often the case that all three criteria are ``reasonable,'' and no criterion can be easily discarded. As a matter of fact, we will illustrate later that all three criteria have been pursued for the classical problem of robust pricing, which has received considerable attention in  Operations Research, Computer Science, and Economics.  Indeed, each criterion is well-motivated and well-grounded from a decision-theoretic perspective. Each criterion, however, could give a markedly different prescription for decision-making. The present paper focuses on the following natural question:

\begin{quote}    
    \textit{Question 1: How good is a prescription derived from one robustness criterion when evaluated against another robustness criterion?}
    \end{quote}
In addition, as mentioned earlier, practitioners are interested in a prescription to follow and may inherently care about the various advantages of the three criteria but not necessarily have a  preference over the criteria. The maximin performance criterion protects against bad scenarios, while the minimax regret and the maximin ratio criteria ensure that the decision-maker does not miss out on the possible upside if a more favorable environment materializes. In turn, we ask whether the decision-maker has to make a choice across the focal criteria. 
\begin{quote}    
   \textit{Question 2: Does there exist a prescription that performs well under all robustness criteria of interest?}
    \end{quote}
In other words, can one have \emph{the best of the three focal robustness criteria}?

We take a first step toward answering these questions by focusing on a well-studied problem across disciplines and uncertainty structures: robust pricing. In a canonical monopoly pricing problem, given the buyers' value distribution $F$, the seller seeks to select a price $p^*$ to maximize revenue, i.e., $p^* \in \arg\max p \bar{F}(p)$ with $\bar{F}(p) = \Pr_{v \sim F}(v \geq p)$.  If the decision maker knows the environment (i.e., the seller knows the value distribution $F$), then the problem is a standard  optimization problem, and the seller simply trades off realized profit and sales volume. In a robust version of the problem, the seller does not know the value distribution $F$ exactly, but knows that it belongs to a class of distributions $\mathcal{F}$. In turn, the decision-maker wants its decision to be ``robustly optimal'' when $F$ can take any value in $\mathcal{F}$.

There is a large stream of literature on robust pricing that we review in \S\ref{sec:lit}. In this paper, we will explore the central questions above across various classical settings. In particular, we will focus on robust pricing settings across various information structures (formalized as classes of uncertainty sets) regarding the partial knowledge of the decision-maker on the distribution of values $F$:  the support; the support and the mean;  the support, the mean and the variance; the support and a quantile. 
Quite notably, we observe that many classes of uncertainty sets have been studied under different robustness criteria across studies. For example, the setting with known support and the first or two moments was studied under the three robustness criteria:  for maximin revenue in \citep{CarrascoLKMMM18-selling-moment-conditions}, for minimax regret in  \citet{shixin-regret}, and  for maximin ratio in \citet{Azar-Micali-parametric,shixin-ratio}.    
For support information, while the maximin revenue criterion is trivial (in that an optimal solution is to price at the lower bound of the support),  both the minimax regret and maximin ratio lead to non-trivial solutions that were studied in parallel in  \citet{BergemannSchlag08} for minimax regret and in \citet{ErenMaglaras10} for maximin ratio. In such settings, each criterion leads to a different pricing mechanisms one could use. This naturally leads to the prescriptive question: which mechanism should one adopt under a particular uncertainty set?

\subsection{Summary of Main Contributions}

At a high level, the present paper initiates the systematic study of the impact of robustness criterion selection, the overfitting that may arise, and the possibility to mitigate such overfitting.   

\paragraph{Unified Analysis.} For each robustness criterion, a corresponding \textit{robustly optimal} mechanism  optimizes against the worst case of that criterion. Therefore, we have focal mechanisms associated with  maximin revenue, minimax regret, and maximin ratio. Each problem is a zero sum game between the seller and Nature, and can be a challenging mathematical program, as the decision-maker solves a high dimensional problem and so does Nature. Building on earlier work and Lagrangian duality, we are able to present a unified analysis of the three focal criteria across various uncertainty structures, which enables us to study the central questions in the present paper.  As we will elaborate in \S\ref{sec:lit},  significant progress has been made on robust pricing  in the last two decades. 
 The goal of the present paper is to bring focus on the criterion in robust decision-making problems and the prescriptive questions this brings to the foreground. In particular, as we will see in a moment, the unification allows to highlight the perils of overfitting to the criterion, but also the possibility of correcting for that (in the context of robust pricing).

\paragraph{Performance of Focal Mechanisms.} We first study how the three mechanisms that emerge as optimal against one criterion fare against another criterion. To do so, we introduce the relative performance of a mechanism for a criterion, defined  as the percentage of best robust performance achievable. The relative worst-case revenue (resp. worst-case ratio) performance of a mechanism is the ratio of its worst-case revenue (resp. ratio) to the maximin revenue (resp. maximin ratio).  The relative worst-case regret  performance of a mechanism is the ratio of its minimax regret to its worst-case regret. All relative performances are always between 0 and 1, and the closer to 1, the closer a mechanism performs to an optimal mechanism for that robustness criterion. 

The forms of the focal mechanisms associated with each criterion as well as their performance clearly depends on the parameter value of the known statistic in question (whether that is mean, variance, median, or lower bound). To concretely answer the paper's two central questions, we focus on a family of instances in which the parameter value varies across a wide range, enumerated in a grid, and evaluate the cross-criteria performance (cf. \S\ref{sec:across}).  For simplicity of exposition, we report the worst case here, but we provide the full set of results in \S\ref{sec:best-results}. To evaluate the relevant performance metrics of interest exactly, we consider a setting in which the valuation distribution can take a finite (but potentially large) number of values.

We start with the prototypical setting of known mean and upper bound (``support'') information since this setting was studied under the three robust criteria separately in the literature. 
 Consider first the mechanism that maximizes worst-case revenue. We show that the minimax regret can be as low as 58\% of the worst-case regret under this mechanism, and the worst-case ratio can be as low as 75\% of the maximin ratio. In other words, while the mechanism is robust based on one criterion, it can perform poorly with regard to the other robust criteria.  Now consider the mechanism that minimizes worst-case regret. Under support and mean information, we show that this mechanism could lead to 45\% of the highest worst-case revenues, and 44\% of the maximin ratio.  Finally, when one optimizes for the worst-case ratio, the prescription performs well with respect to regret, achieving 93\% of the optimal regret, but the mechanism revenue can be as low as  68\% of the optimal worst-case revenue (Proposition~\ref{prop:across-mean} and Table~\ref{table:WC-Rel-robust-pricing-F-mu}). 
 
 While it is well known that the worst-case revenue criterion could be overly pessimistic when little information is available, the above highlights that even with more knowledge of the distribution (the mean), the worst-case revenue criterion leads to a prescription that does not fare well when one accounts for the upside of a decision through a benchmark. Similarly, while the regret and ratio criteria are thought to mitigate the pessimism of the worst-case revenue criterion, the prescriptions they lead to are also fragile in the sense that they can lead to very poor performance across other criteria. 

The negative results above are not hand-picked examples. They are representative across uncertainty sets.  We show that all three focal mechanisms can perform very poorly against another criterion across the various information settings described above. We summarize these cross-criteria performance results in Table~\ref{table:WC-Rel-robust-pricing-F-mu} (mean information), Table~\ref{table:WC-Rel-robust-pricing-F-mu-sigma} (mean and variance information), Table~\ref{table:WC-Rel-robust-pricing-F-q} (median information), and Table~\ref{table:WC-Rel-robust-pricing-F-a} (lower bound information), all in Section~\ref{sec:across}. 

\paragraph{The best of many criteria.} Having derived the above results, we ask if it is possible to perform well across all three criteria. We propose to capture the best-of-many-criteria performance as the \textit{relative performance guarantee} that holds \textit{uniformly over all criteria} (cf. \eqref{eqn:relperf-all-def}, \eqref{eqn:c_star} and \eqref{eqn:relperf-decoupled}). In other words, it is worst case not only over environments in the uncertainty set, but also over criteria. We focus on the \textit{relative performance} of each criterion (that is, divided by the optimal performance if we were to solely focus on that criterion) to normalize the performance metrics to between 0 and 1 and make them comparable across criteria. This performance metric has an intuitive interpretation: if we have a mechanism with a guarantee of, say, 80\%, it means the mechanism achieves at least 80\% of the maximin revenue, at most 1/80\% of the munimax regret, and at least 80\% of the maximin ratio simultaneously.

For each instance, we can compute the optimal uniformly robust mechanism and its performance as follows. We first compute the maximin  revenue, minimax regret and maximin ratio. For each $c \in [0,1]$, we can formulate a linear feasibility program that checks whether a mechanism has worst-case revenue, regret, and ratio within a factor of $c$ of the optimal, and returns such a mechanism if one exists. The best performance and mechanism corresponds to the largest feasible such $c$, which can be computed via binary search (cf. Proposition~\ref{prop:rho-lp}).

Returning to the broad classes of uncertainty sets described earlier, we show that for each class of uncertainty sets it it is possible to design mechanisms that can do well across the three focal criteria (Theorem~\ref{thm:factor-c-lb}).  Table~\ref{table:relperf-all-vs-focal-intro} below summarizes, for each type of uncertainty sets (where we fix an upper bound of the support), the performance across criteria of the uniformly robust optimal mechanism that we design and compare it to that of the three focal criteria.
\begin{table}[h!]
\centering
\begin{tabular}{c || c || c | c | c } 
 %\hline
  Additional & Uniformly Robust & \multicolumn{3}{c}{Focal Mechanisms} \\
  \cline{3-5}
 Information &  Mechanism & revenue & regret & ratio  \\ [0.5ex] 
 \hline
 \hline
 mean & \textbf{92\%} & 58\% & 44\% & \textbf{68\%}  \\
 %\hline
  mean and variance & \textbf{86\%} & 51\% & 49\% & \textbf{71\%} \\ %\hline 
  median & \textbf{61\%} & 34\% & 0\% & \textbf{41\%}\\ %\hline
    lower bound & \textbf{58\%} & \textbf{33\%} & 0\% & 31\% \\ 
 \hline
\end{tabular}

\caption{Worst-case (across instances) relative performance across all criteria of the uniformly optimal mechanism, compared to that of the three focal mechanisms: maximin revenue mechanism (``revenue''), minimax regret mechanism (``regret'') and maximin ratio mechanism (``ratio''). The performances of the uniformly robust optimal mechanism as well as the best among all focal mechanisms are bolded for emphasis.}
\label{table:relperf-all-vs-focal-intro}
\end{table}

These results provide a broad picture of achievability results. For example, with mean information, quite strikingly, there is a mechanism that guarantees relative performance at least 92\% across the three criteria, even though none of the three focal mechanisms could achieve performance exceeding 68\%. Across settings, we observe that there is significant room to increase robustness to criteria. We emphasize that the performance numbers shown in Table~\ref{table:relperf-all-vs-focal-intro} are worst-case across instances. For some instances, the performance of the optimal uniformly robust mechanism can be much higher than $61\%$. For example, when the median is 0.5, 0.6, or 0.7, the guarantees are 93\%, 98\%, and 93\%, respectively.

Our results demonstrate that, in some sense, decision-makers need not compromise by choosing one robust criterion and run the risk of overfitting to it.  It is possible to do well against all three criteria, despite their distinctive nature, which may be a practically appealing desiderata in robust decision-making. Finally, while we focus on the robust pricing problem, we hope our work ushers in the study of criteria-robustness across other decision problems.

\subsection{Related Work}\label{sec:lit}

Our work is related to several streams of literature.

\paragraph{Robust Pricing.} Our work is closely related to the literature on robust monopoly pricing. 

With only knowledge of the support, \citet{BergemannSchlag08} studies minimax regret and \citet{ErenMaglaras10} studies maximin ratio.

With knowledge of some moment(s), \citet{kos-messner} and \citet{CarrascoLKMMM18-selling-moment-conditions} study the maximin revenue criterion, while \citet{shixin-regret} studies the minimax regret criterion. For the maximin ratio,   \citet{Azar-Micali-parametric}  gives an optimal deterministic mechanism when the mean and variance are known,  while \citet{shixin-ratio} gives an optimal randomized mechanism for general distributions and arbitrary moment information.  More closely related to our paper is \citet{ChenHuPerakis22-distribution-free-pricing}. While not their main focus, the paper is one of the exceptions that illustrates some form of cross-criteria performance. In particular, the authors, for the specific case of known mean and variance,  study the optimal deterministic strategy  under the three criteria and evaluate them under the worst-case distribution of the maximin revenue criterion. As such, \citet{ChenHuPerakis22-distribution-free-pricing} also suggests that the optimal prices for regret and ratio may not perform well for the revenue criterion.  In contrast, we systematically consider the performance transfer of all pairs of old and new criteria, study general randomized mechanisms and many classes of uncertainty sets, and evaluate the worst-case specific to each transfer pair, and the possibility of being robustly optimal across all three criteria.

 Under the maximin revenue criterion, \citet{chen-hu-wang-screening-dual} provides an overview of some studies and connections across problem classes. 

With knowledge of  a quantile, \citet{ErenMaglaras10}  studies maximin ratio under general distributions with bounded support. In the absence of support information but with  shape constraints (regular or monotone hazard rate),  \citet{AzarDMW13} gives a constant-factor-approximation for deterministic mechanisms, and \citet{AllouahBahamouBesbes21-pricing-single-point} derive the maximin ratio. These results highlight the value of a single quantile. \citet{shixin-simple-menus} studies the power of simple mechanisms (randomization over a few prices) for general distributions with arbitrary quantiles and support information.   With quantile information, other robust approaches have also been considered. \citet{Chehrazi2010} considers data-driven robust optimization with splines.

This work is also related to the broader robust mechanism design literature with  multiple products \citep{Carroll17-robust-screening,GravinLu18, KocyigitRK21-old, ChenHuPerakis22-distribution-free-pricing}, or multiple buyers \citep{ AllouahBesbes20, HartlineJohnsenLi20-benchmark-design, RougTalg2019, our-paper-ec,our-paper-ec-a-b,Suzdaltsev20-dr-auction,BachrachTalgamCohen22}.  %

\paragraph{Robust Decision Making Under Uncertainty.} Our work is also related to robust decision making, robust optimization and distributionally robust optimization. As discussed in the introduction, multiple robust optimality criteria have been studied in decision theory and related areas: worst-case performance \citep{Wald45}, regret (or absolute regret) \citep{Savage51,Savage54}, and ratio (or competitive ratio, or relative regret) \citep{BorodinElYaniv}. In some cases, some of these criteria may be inappropriate, but in many other cases, all of them are reasonable, and our work has shown that in the case of pricing, it is not necessary to commit only to one criterion. While these three aforementioned criteria are most common, other possible criteria include the pessimism-optimism index (the weighted average of the best case and the worst case, due to \citet{Hurwicz1951}) and principle of insufficient reason (assign every contingency equal probability, due to Laplace \citep{HowsonUrbach1989}). \citet{Stoye2011} unifies and extends the axiomatic treatment of minimax regret. 

All the aforementioned criteria deal with \textit{non-Bayesian} uncertainty, when the decision maker does not know enough about the likelihood of different outcomes. These are in contrast to the Bayesian approach where the uncertainty is modeled as a known probability distribution over outcomes. \citet{vonNeumannMorgenstern1944} laid the foundation for expected utility approach under Bayesian uncertainty, which was further developed by \citet{harsanyi1955cardinal,Savage54,arrow1951alternative}. These decision-theoretic models (for both non-Bayesian and Bayesian uncertainty) can be viewed either as \textit{normative}, giving a prescription for how one should make decisions, or \textit{positive}, describing and predicting how humans actually make decisions. The Allais and Ellsberg paradoxes \citep{allais1953paradox,ellsberg1961paradox} cast doubt on the validity of the expected utility model as a positive description in some cases, giving rise to alternative models such as prospect theory \citep{kahneman1979prospect}. There are also intermediate approaches that combine both the Bayesian and non-Bayesian aspects of uncertainty, such as ambiguity aversion, or maximin expected utility/minimax expected regret with multiple priors \citep{gilboa1989maxmin}.

\paragraph{Robust Optimization.}  The robust optimization \citep{bertsimas-robust-survey,BertsimasDenHertog} and distributionally robust optimization \citep{dro-review} literatures tend to focus more on deriving computationally efficient algorithms in different settings, with the emphasis on robustness over unknown environments and the worst-case performance criterion formulation is often adopted. This work highlights an orthogonal and often overlooked aspect that robustness across criteria is also important. 

 The robust optimization framework has been used in many application areas including newsvendor \citep{roels-perakis-newsvendor}, assortment optimization \citep{rusmevichientong-robust-assortment-optimization}, scheduling \citep{daniels1995-robust-scheduling}, process flexibility \citep{wang2022-robust-process-flexibility}, and security games \citep{Kiekintveld-robust-security-games}. We consider the pricing problem in this paper as a starting point, but our proposed best-of-many-criteria framework is broadly applicable to robust decision-making problems in general, and it is worth revisiting these problems to see under which problems and information structure the prescriptions robustly transfer, and which ones admit solutions that are uniformly robust.

\section{Best of Many Robustness Criteria  Formulation: Robust Pricing}\label{sec:problem-formulation}

We study the mechanism design problem in which a seller wants to sell a single product to a buyer. The buyer has a value for the product drawn independently from a distribution $F$. (Equivalently, there is a continuum of buyers, and $F$ represents the distribution of buyer valuations.) 

If the seller knew the distribution $F$ precisely, then classical results \citep{Myerson81,RileyZeckhauser83} tell us that the deterministic posted-price mechanism is optimal, with optimal revenue given by 
\begin{align}
\textnormal{OPT}(F) = \max_{p} p \bar{F}(p) , \label{eqn:opt-revenue}
\end{align}
with $\bar{F}(p) = \Pr_{v \sim F}(v \geq p)$. In other words, this problem is equivalent to the monopoly pricing problem. It is also known as the screening problem in the mechanism design literature.\jaedit{\footnote{We will assume throughout that if the valuation equals the price, then the buyer buys the product.}}

However, here the seller does not know $F$. The seller only knows the values are supported on a set $\gridval \subseteq [a,b]$ as well as, potentially, the first $I$ moments $m_i$ for moment indices $i \in \mathcal{I} \equiv \{0,1,\dots,I\}$ with $m_0 = 1$, as well as $J$ quantiles: the $r_j$ quantile is $q_j$ for quantile indices $j \in \mathcal{J} \equiv \{1,2,\dots,J\}$. 
 Formally, let $\gridval \subseteq [a,b]$ be the set of possible values and $\Delta(\gridval)$ denotes the set of distributions supported on $\gridval$, then $F$ belongs to the uncertainty set $\mathcal{F}$ given by
\begin{align}\label{eqn:uncertainty-set}
    \mathcal{F} = \left\{ F \in \Delta(\gridval) : \int_{v \in \gridparam} v^{i} dF(v) = m_i \text{ for } i \in \mathcal{I}, \text{ and } \bar{F}(r_j) = q_j \text{ for } j \in \mathcal{J}  \right\}.
\end{align}
This form of the uncertainty set is very general and is in line, e.g., with the one recently considered by \citet{shixin-simple-menus} to study simple mechanisms, and captures many different forms of partial information that the seller may have about the distribution. In particular, it captures various prototypical forms of uncertainty: knowing support only, support and mean, support as well as mean and variance, support and median, correspond to $(I,J) = (0,0),(1,0),(2,0),(0,1)$ with $r_j = 1/2$ in the median case. 

The moments are commonly used summary statistics for the distribution, especially the first two moments. The mean $\mu$ and variance $\sigma^2$ of the distribution correspond to $m_1 = \mu, m_2 = \mu^2 + \sigma^2$. The mean and variance, respectively, capture the central tendency and dispersion of the valuation distribution and are commonly estimated and used by business managers \citep{ChenHuPerakis22-distribution-free-pricing}. 

 The median is also a common measure for central tendency and it corresponds to $r_j = 1/2$.  Apart from a priori estimates of quantiles, quantile data are also a natural form of \jaedit{partial information} that could potentially be obtained though transaction data and associated conversion rates at historical prices \citep{AllouahBahamouBesbes21-pricing-single-point}.

\paragraph{Seller's mechanisms.} We will now consider the seller's mechanism class. By \citet{Myerson81}, any incentive compatible mechanism in this one-seller-one-buyer case is equivalent to randomized pricing. Given this, henceforth, we assume that the seller optimizes over the distributions of posted prices $\Phi(\cdot)$ over the set $\mathcal{M}$ of valid CDFs.  Note that because we know the values are in $\gridval$, any optimal price is also in $\gridval$, so it is sufficient to consider a distribution of prices over $\gridval$. In particular, if $\gridval = \{v_0,v_1,\dots,v_{K}\}$ is a discrete set, then the seller optimizes over distributions with weight $\phi_i$ on $v_i$ for $i\in [K]$ with $\phi_i \geq 0, \sum_{i} \phi_i = 1$.

\paragraph{Robust criteria.} We are now ready to define formally the three robust optimality criteria, discussed  in the introduction. Given a mechanism $\Phi$  and a distribution $F$, we define the  revenue, regret, and ratio as
\begin{align*}
 \revenue(\Phi,F) &:= \int s \left(\int_{v \ge s} dF(v)\right) d\Phi(s) = \int\int_{s \leq v} s d\Phi(s) dF(v),
 \\
    \regret(\Phi,F) &:= \textnormal{OPT}(F) -  \revenue(\Phi,F), \\
    \ratio(\Phi,F)  &:= \frac{ \revenue(\Phi,F) }{\textnormal{OPT}(F)}.
\end{align*}
Note that if we want the revenue to be high compared to the benchmark, we want regret to be low and ratio to be high. Given that seller does not know $F$, a robust formulation proceeds to assume that Nature selects  $F$ in an adversarial fashion. Therefore, given a fixed mechanism $\Phi$, the worst case performances of this mechanism in these 3 metrics are given by
\begin{align*}
     \min_{F \in \mathcal{F} } \revenue(\Phi,F), \: \max_{F \in \mathcal{F} } \regret(\Phi,F), \: \textnormal{ and } \: \min_{F \in \mathcal{F} } \ratio(\Phi,F).
\end{align*}

The criterion of worst-case revenue protects against the possibility of environments (distributions) with low revenues. On the other hand, both the worst-case regret and worst-case ratio account for what would have been possible  ($\textnormal{OPT}(F)$), and evaluate a mechanism based on the actual environment that materializes. A good mechanism for such criteria should not only protect against environments with low revenue opportunities, but also ensure that the seller does well in environments with high revenue opportunities.

\paragraph{Focal mechanisms.} For each criterion, the corresponding \textit{robust} mechanism  optimizes the worst case of the associated criterion. Therefore, we have maximin revenue $\theta^*_\revenue$, minimax regret $\theta^*_\regret$, and maximin ratio $\theta^*_\ratio$ defined as 
\begin{align*}
  %\maximinrevenue(\mathcal{F}) &:= 
  \theta^*_\revenue(\mathcal{F})  &:= \max_{\Phi \in \mathcal{M}} \min_{F \in \mathcal{F} } \revenue(\Phi,F), \\
    %\minimaxregret(\mathcal{F}) &:= 
    \theta^*_\regret(\mathcal{F})  &:= \min_{\Phi \in \mathcal{M}} \max_{F \in \mathcal{F} } \regret(\Phi,F), \\
    %\maximinratio(\mathcal{F}) &:= 
    \theta^*_\ratio(\mathcal{F})  &:=  \max_{\Phi \in \mathcal{M}} \min_{F \in \mathcal{F} } \ratio(\Phi,F).
\end{align*}

 Assuming such mechanisms exist, we denote by $\mechrevenue$ a maximin revenue optimal  mechanism, $\mechregret$ a minimax regret optimal  mechanism,  and $\mechratio$ a maximin ratio optimal mechanism.\footnote{The mechanisms and their performance values depend on $\mathcal{F}$, but sometimes we omit this dependence for simplicity.}

\paragraph{Performance across robust criteria.} We are now ready to formally define the objects that will allow to answer the first question we asked in the introduction, namely, how robust is the mechanism derived from one criterion for other criteria. Given that three criterion-specific performance values are expressed on different scales, to be able to compare across settings, we will measure robust performance across criteria in relative terms. More specifically, we will measure the worst-case performance of a mechanism compared to the performance of the optimal mechanism for that criterion.

Fix an uncertainty set $\mathcal{F}$. Formally, the relative performances of a given mechanism $\Phi$ on the revenue, regret,  and ratio criteria are defined by
\begin{align*}
        \relperf(\Phi, \revenue, \mathcal{F}) %&= \frac{\worstrevenue(\Phi,\mathcal{F})}{\worstrevenue(\mechrevenue, \mathcal{F})}, \\ 
        &= \frac{\worstrevenue(\Phi,\mathcal{F})}{\theta^*_\revenue}, \\
    \relperf(\Phi, \regret, \mathcal{F}) %&= \frac{\worstregret(\mechregret, \mathcal{F}) }{\worstregret(\Phi, \mathcal{F})}, \\ 
    &= \frac{ \theta^*_\regret }{\worstregret(\Phi, \mathcal{F})}, \\
    \relperf(\Phi, \ratio, \mathcal{F}) %&= \frac{\worstratio(\Phi,\mathcal{F})}{\worstratio(\mechratio, \mathcal{F})}. %&
    &= \frac{\worstratio(\Phi,\mathcal{F})}{\theta^*_\ratio}.
\end{align*}
Note that the definition of relative performance is different (the reciprocal) for regret, compared to the two other criteria. This is simply because a decision maker tries to minimize regret, while it tries to maximize revenue and the ratio. When defined as above, the relative performance of a mechanism is always between 0 and 1 and the higher the relative performance the more robust the mechanism is according to that criterion. This object enables us to quantify the robustness of each focal mechanism $\mechrevenue, \mechregret,  \mechratio$ against other criteria.

\paragraph{The best of many robustness criteria.} While the focal point of the literature has been on robustness to environments as captured by the uncertainty set $\mathcal{F}$,  here we explore robustness to environments \emph{as well as} criteria. 
For a given mechanism $\Phi$,  we can evaluate its relative performance with respect to each criterion. Rather than fixing the mechanism to be one of the focal mechanisms, we propose to  directly optimize over all mechanisms for the best relative performance guarantee over all three criteria. Formally, we define the relative performance of a mechanism $\Phi$ over all criteria as
\begin{align}\label{eqn:relperf-all-def}
    \relperf(\Phi,\all,\calF) %= \min_{F \in \mathcal{F}} \min\left\{ \frac{\revenue(\Phi,F)}{  \maximinrevenue(\calF) }, \frac{ \minimaxregret(\calF)}{\regret(\Phi,F)}, \frac{\ratio(\Phi,F)}{\maximinratio(\calF)}   \right\}, 
  = \min_{F \in \mathcal{F}}   \min\left\{ \frac{\revenue(\Phi,F)}{\theta^*_\revenue} , \frac{\theta^*_\regret}{\regret(\Phi,F)}, \frac{\ratio(\Phi,F)}{\theta^*_\ratio}\right\},
\end{align}
and the \textit{multi-criterion robust approximation factor} $c^*(\mathcal{F})$ as
\begin{align} \label{eqn:c_star}
    c^*(\mathcal{F}) = \max_{\Phi \in \mathcal{M}} \relperf(\Phi, \all, \mathcal{F}).
\end{align}
Denote by $\mechall$ a mechanism that achieves the optimal factor $c^*$.
The definition of $c^*$ implies that the mechanism $\mechall$  simultaneously achieves a fraction $c^*$ of the tailored (optimal) mechanism for all three criteria, and it is the highest constant possible. If $c^*(\mathcal{F})$ is close to 1, then the uncertainty set $\mathcal{F}$ allows us to perform almost as well as any criterion-tailored mechanism. 

\paragraph{Remark.} We note that it is easy to see that $\relperf(\Phi,\all,\calF)$ can be equivalently written as
    \begin{align}\label{eqn:relperf-decoupled}
        \relperf(\Phi,\all,\calF) %=  \min\Big\{\relperf(\Phi, \revenue, \mathcal{F}),      \relperf(\Phi, \regret, \mathcal{F})  , \relperf(\Phi, \ratio, \mathcal{F})   \Big\}.
    = \min\left\{ \frac{\min_{F \in \calF} \revenue(\Phi,F) }{ \theta^*_\revenue}, \frac{\theta^*_\regret}{ \max_{F \in \calF} \regret(\Phi,F) } , \frac{ \min_{F \in \calF} \ratio(\Phi,F) }{ \theta^*_\ratio }  \right\}.
    \end{align}
We formalize this in Proposition \ref{prop:relperf-equivalent}, and we will use this fact to develop various alternative representations.

\section{The Best of Many Robustness Criteria Through Linear Programs}\label{sec:lp-formulation}

\subsection{Unified Evaluation  Through $\lambda$-Regret} 
We first observe that, as noted in \cite[Proposition 1]{our-paper-ec-a-b}, one can unify the maximum revenue, minimax regret,  and maximum ratio criteria by focusing on the $\lambda$-regret. Fix $\lambda \in [0,1]$. We define the $\lambda$-regret of a given mechanism $\Phi$, $R_{\lambda}^{\Phi}(\mathcal{F})$, and the minimax $\lambda$-regret, $R_{\lambda}^*(\mathcal{F})$,  as follows.
\begin{align*}
R_{\lambda}^{\Phi}(\mathcal{F}) &:=  \max_{F \in \mathcal{F}} \left[ \lambda \textnormal{OPT}(F) - \revenue(\Phi,F) \right],\\
    R_{\lambda}^*(\mathcal{F}) &:= \min_{\Phi \in \mathcal{M}} \max_{F \in \mathcal{F}} \left[ \lambda \textnormal{OPT}(F) - \revenue(\Phi,F) \right].
\end{align*}
The following proposition implies that if one can calculate the minimax $\lambda$-regret for every $\lambda \in [0,1]$, then one can immediately derive the robust value associated with all three criteria.

\begin{proposition}[\citep{our-paper-ec-a-b}]\label{prop:minimax-regret-exp}
Fix a class of distributions $\calF$ as in equation (\ref{eqn:uncertainty-set}). Suppose that the mapping $\lambda \mapsto R_{\lambda}^*(\mathcal{F})$ is known. Then one can directly obtain the maximin revenue, minimax regret and maximin ratio as follows:  $\textnormal{MinimaxRegret}(\mathcal{F}) = R_{1}^*(\mathcal{F})$,  $\textnormal{MaximinRevenue}(\mathcal{F}) = - R_{0}^*(\mathcal{F})$,  and $\textnormal{MaximinRatio}(\mathcal{F})$ is the largest constant $\lambda \in [0,1]$ such that $R_{\lambda}^*(\mathcal{F}) \leq 0$.\jaedit{\footnote{While we focus on $\lambda$-regret as a unifying framework in the main text, we note that in the case where ratio is the objective, it is also possible to formulate a single linear program to compute the maximin ratio directly, which is more computationally efficient compared to the line search approach where we solve a series of linear programs, one for each value of $\lambda$, to find the highest $\lambda$ such that minimax $\lambda$-regret is nonpositive as shown by Proposition~\ref{prop:minimax-regret-exp}. The two approaches, by necessity, give the same answer. See Appendix~\ref{sec:maximin-ratio-lp} for details.}}
\end{proposition}

The following proposition shows that we can compute the $\lambda$-regret of any mechanism and the minimax $\lambda$-regret via solving  infinite linear programs.
\begin{proposition}[$\lambda$-regret Linear Programs] \label{prop:minimax-lmbd-regret-lp}
Fix a class of distributions $\calF$ as in equation (\ref{eqn:uncertainty-set}) and a mechanism $\Phi$.  Its worst-case $\lambda$-regret $R_{\lambda}^{\Phi}(\mathcal{F})$ can be obtained from the following linear program
\begin{equation}
\begin{aligned}
    \min_{\theta, \alpha(\cdot), \beta(\cdot) } & \:  \theta \\
    \textnormal{ s.t. }  &  \theta \geq  \sum_{i \in \mathcal{I}} \alpha_{i}(p) m_i + \sum_{j \in \mathcal{J}} \beta_{j}(p) q_j  \quad \forall p \in \jaedit{\gridval} \\
    &  \lambda p \1(v \geq p) - \int_{s \leq v} s d\Phi(s)  - \sum_{i \in \mathcal{I} } \alpha_{i}(p) v^i - \sum_{j \in \mathcal{J}} \beta_{j}(p) \1(v \geq r_j) \leq 0 \quad \forall v, p \in \jaedit{\gridval}.
\end{aligned}
\label{eqn:lp-regret} \tag{$\mathcal{LP}^{\Phi}_{\lambda}$} %  %\tag{Regret-LP}
\end{equation}

Furthermore, the minimax $\lambda$-regret value $R^*_{\lambda}(\mathcal{F})$  is given by the objective value of the following linear program
\begin{equation}
\begin{aligned}
    \min_{\substack{\Phi(\cdot), \theta\\ \alpha(\cdot), \beta(\cdot)} } & \: \theta  \\
   \textnormal{ s.t. } & \theta \geq  \sum_{i \in \mathcal{I}} \alpha_{i}(p) m_i + \sum_{j \in \mathcal{J}} \beta_{j}(p) q_j  \quad \forall p \in \jaedit{\gridval} \\
   &  \lambda p \1(v \geq p) - \int_{s \leq v} s d\Phi(s)  - \sum_{i \in \mathcal{I} } \alpha_{i}(p) v^i - \sum_{j \in \mathcal{J}} \beta_{j}(p) \1(v \geq r_j) \leq 0 \quad \forall v, p \in \jaedit{\gridval} \\
   &  \Phi \textnormal{ is a CDF,}
\end{aligned}
\label{eqn:lp-minimax} \tag{$\mathcal{LP}^{*}_{\lambda}$} %\tag{Minimax-LP}
\end{equation}
and the solution $\Phi^*$ of this linear program is a mechanism that achieves the optimal value $R^*_{\lambda}(\mathcal{F})$.
\end{proposition}

Note that the condition that $\Phi$ is a CDF is a linear constraint. Note also that if $\gridval = \{v_0 < \dots <v_K\}$ is a finite set, then it is sufficient to optimize over $\phi_0,\dots,\phi_K$ such that $\Phi$ has weight $\phi_i$ on $v_i$ with $\phi_i \geq 0, \sum_{i \in [N]} \phi_i = 1$, and for $v = v_{k}$, we can write $\int_{s \leq v} s d\Phi(s) = \sum_{i=0}^{k} s \phi_i$. \eqref{eqn:lp-minimax} then becomes a linear program with $\Theta(K)$ variables and $\Theta(K^2)$ constraints. 

The proof of Proposition~\ref{prop:minimax-lmbd-regret-lp} relies on some transformations and linear programming duality. We describe the high level ideas for the result leading to \eqref{eqn:lp-minimax}.  We first note that  $\textnormal{OPT}(F) \ge \int_{\jaedit{\gridval}} \1(v \geq p) dF(v)$ for all $p \in \jaedit{\gridval}$ with the equality achieved at some $p$.  This, together with the epigraph formulation, leads to the fact that the minimax $\lambda$-regret is the value of the program
\begin{align*}
    \min_{\theta, \Phi(\cdot)} & \: \theta  \\
  \text{ s.t. } &  \theta \geq \max_{F \in \mathcal{F}} \left\{ \lambda p \int_{\jaedit{\gridval}} \1(v \geq p) dF(v) - \int \int_{s \leq v} s d\Phi(s) dF(v) \right\} \quad \forall p \in \jaedit{\gridval}\\
  &  \Phi \textnormal{ is a CDF}.
\end{align*}
We have replaced the $\opt$ term with a continuum of constraints indexed by $p$, which is more tractable. The next step is to convert the $\max_{F \in \mathcal{F}}$ term into a minimization problem through LP duality, enabling one to add the dual variables to the optimization problem while eliminating the minimization in the constraint (since the $\theta \geq \min$ constraint simply means there exists a set of dual variables that make the constraint feasible). We use the form of the uncertainty set $\mathcal{F}$ in an essential way here because both moment and quantile constraints are linear in the distribution variable $dF$, and so is the objective. The max problems are indexed by $p$, so if we let $\alpha_{i}(p)$ and $\beta_{j}(p)$ be the dual variables of the $i$'th moment and $j$'th quantile constraints of the $p$-problem, we get exactly the linear program in question. The full proof is given in Appendix~\ref{app:sec:lp-formulation}.

\paragraph{Remark.} The formulation and derivation of \eqref{eqn:lp-regret} and  \eqref{eqn:lp-minimax} is inspired by techniques in \citet{roels-perakis-newsvendor} in the context of minimax regret analysis for a newsvendor problem,   and are very close to those of \citet{shixin-simple-menus} in the context of pricing. We do not claim them as technical contributions in their own right.  
Rather, our main contribution is to unify different criteria and uncertainty sets, and define and analyze a \textit{quantitative notion of approximate optimality across many criteria}. We view these linear programming formulations as analytical tools for us to unify all three focal criteria.

\subsection{Cross-Criterion Performance} 

 We are particularly interested in how a mechanism that is robustly optimized for one criterion performs under another. That is, with a given uncertainty $\mathcal{F}$, we want to evaluate:
\begin{itemize}
    \item the worst-case regret and  ratio of a mechanism $\mechrevenue$ that achieves maximin revenue.
    \item the worst-case revenue and  ratio of a mechanism $\mechregret$ that achieves minimax regret.
    \item the worst-case revenue and regret of a mechanism $\mechratio$ that achieves maximin ratio.
\end{itemize}

 One way to go about this is to  compute an optimal mechanisms for one criterion, say $\mechrevenue$ for revenue, then use \eqref{eqn:lp-regret} to compute the worst case regret and ratio. But this approach has an issue. The evaluated performance will only be for a specific optimal mechanism, but this would not say anything about  other potentially optimal mechanisms. In particular poor performance under another criterion would not preclude that another optimal mechanism could perform well. We aim for a  result that applies to \textit{any} optimal mechanism.  Therefore, we instead use a bilevel approach. We have observed earlier that all criteria can be unified in a $\lambda$-regret framework, so we will use the $\lambda$-regret language throughout. For any $\lambda_\old \in [0,1]$, let  $r_\old$ denote the corresponding $\lambda_\old$-regret.  A mechanism $\Phi$ is an optimal mechanism for the previous ``old'' criterion if and only if $R_{\lambda_\old}^{\Phi}(\calF) \leq r_\old$.  To evaluate \textit{any} optimal mechanism against a ``new'' criterion characterized by $\lambda_{\new}$, we compute the minimax $\lambda_\new$-regret objective with the aforementioned constraint on $\Phi$. In other words, we want to solve
\begin{equation}
\begin{aligned}
    R_{\lambda_{\new}}^*(\mathcal{F},(\lambda_{\old},r_{\old})) = \min_{\Phi \in \mathcal{M}_\old}  R_{\lambda_\new}^{\Phi}(\calF),  \textnormal{ where } \mathcal{M}_\old = \left\{ \Phi :  R_{\lambda_\old}^{\Phi}(\calF) \leq r_\old \right\}.
\end{aligned}
%\label{eqn:constrained-regret} \tag{Constrained-Regret}
\end{equation}
We call this quantity the cross regret. The following proposition shows that this problem can be written as a standard linear program.
\begin{proposition}[Linear Program for Cross Regret] \label{prop:across-lp}
Fix a class of distributions $\calF$ as in equation (\ref{eqn:uncertainty-set}),  $\lambda_{\new}, \lambda_{\old} \in [0,1]$ and $r_{old} \in \mathbb{R}$. The cross regret $R_{\lambda_{\new}}^*(\mathcal{F},(\lambda_{\old},r_{\old}))$ can be obtained through the following linear program \eqref{eqn:lp-across}.
\begin{equation}
\begin{aligned}
    \min_{\substack{\Phi,\theta\\ \alpha_\new, \beta_\new\\ \alpha_\old, \beta_\old }} & \: \: \:  \theta  
    \\
    \textnormal{ s.t. }  &  \theta \geq  \sum_{i \in \mathcal{I}} \alpha_{\new,i}(p) m_i + \sum_{j \in \mathcal{J}} \beta_{\new,j}(p) q_j  \quad \forall p \in \jaedit{\gridval} \\
    &  \lambda_{\new} p \1(v \geq p) - \int_{s \leq v} s d\Phi(s)  - \sum_{i \in \mathcal{I} } \alpha_{\new,i}(p) v^i - \sum_{j \in \mathcal{J}} \beta_{\new,j}(p) \1(v \geq r_j) \leq 0 \quad \forall v, p \in \jaedit{\gridval} \\
    &r_\old \geq  \sum_{i \in \mathcal{I}} \alpha_{\old,i}(p) m_i + \sum_{j \in \mathcal{J}} \beta_{\old,j}(p) q_j  \quad \forall p \in \jaedit{\gridval} \\
    &  \lambda_{\old} p \1(v \geq p) - \int_{s \leq v} s d\Phi(s)  - \sum_{i \in \mathcal{I} } \alpha_{\old,i}(p) v^i - \sum_{j \in \mathcal{J}} \beta_{\old,j}(p) \1(v \geq r_j) \leq 0 \quad \forall v, p \in \jaedit{\gridval} \\
   &  \Phi \textnormal{ is a CDF.}
\end{aligned}
\label{eqn:lp-across} \tag{$\mathcal{LP}^{*,(\lambda_{\old},r_{\old})}_{\lambda_{\new}}$}%\tag{Across-LP}
\end{equation}

\end{proposition}

Proposition~\ref{prop:across-lp} is derived by using \eqref{eqn:lp-regret} to convert the worst case $\lambda_\new$-regret problem in the objective and worst case $\lambda_\old$-regret problem in the $\mathcal{M}_\old$ constraint into minimization problems with dual variables $(\alpha_\new,\beta_\new)$ and $(\alpha_\old,\beta_\old)$ respectively. The resulting minimization problems can then  be absorbed into the overall optimization problem with additional variables. As before, if $\gridval$ is a finite set of size $K$, then \eqref{eqn:lp-across} is an LP with $\Theta(K)$ variables and $\Theta(K^2)$ constraints.

\paragraph{Procedure to evaluate cross-criteria performance.} The old and new criteria translate into linear program parameters as follows. If the old criterion is regret, we first compute $\theta^*_\regret$ using \eqref{eqn:lp-minimax} with $\lambda = 1$, and set $\lambda_\old = 1, r_\old = \theta^*_\regret$. If the old criterion is revenue, we first compute $\theta^*_\revenue $ using \eqref{eqn:lp-minimax} with $\lambda = 0$, and set $\lambda_\old = 0, r_\old = -\theta^*_\revenue$. If the old criterion is ratio, we first compute \eqref{eqn:lp-minimax} for different values of $\lambda$, and using a binary search, obtain $\theta^*_\ratio$; then we  set $\lambda_\old = \theta^*_\ratio, r_\old = 0$.  For the new criterion, one would set $\lambda_\new = 1$ for regret, $\lambda_\new = 0$ for revenue, and for ratio, one would conduct a binary search over $\lambda_\new$ until we get the value of $\lambda_\new$ such that the program \eqref{eqn:lp-across} evaluates to zero.

\subsection{Linear Programs for the  Best of Many Robustness Criteria}  We now give a linear programming formulation of the best of many criteria. To that end, we first slightly generalize the problem and ask: given a tuple $( \theta_\revenue, \theta_\regret,\theta_\ratio)$, whether there exists a mechanism $\Phi$ such that  
\begin{align*}
\min_{F \in \calF} \revenue(\Phi,F) &\geq \theta_\revenue \\
\max_{F \in \calF} \regret(\Phi,F) &\leq \theta_\regret \\
\min_{F \in \calF} \ratio(\Phi,F) &\geq \theta_\ratio  .  
\end{align*}

\noindent We will say that the mechanism $\Phi$ \textit{achieves} $( \theta_\revenue, \theta_\regret,\theta_\ratio)$ if the conditions are satisfied.

In terms of $\lambda$-regret, these conditions are equivalent to the conditions that $R_{0}^{\Phi}(\mathcal{F}) \leq -\theta_\revenue$, $R_{1}^{\Phi}(\mathcal{F}) \leq \theta_\regret$, and $R_{\theta_\ratio}^{\Phi}(\mathcal{F}) \leq 0$.
 By leveraging LP duality in an analogous way as in Proposition~\ref{prop:minimax-lmbd-regret-lp} to convert maximum over $F$ in each worst case regret expression to a minimum over dual variables, we obtain that characterizing the feasibility of $(\theta_\revenue, \theta_\regret,\theta_\ratio)$ is equivalent to the following linear feasibility problem.

\begin{proposition}\label{prop:rho-lp}
There exists a mechanism $\Phi$ that achieves $(\theta_\revenue, \theta_\regret,\theta_\ratio)$ if and only if the following linear problem in variables $\alpha_\revenue, \alpha_\regret, \alpha_\ratio,\beta_\revenue, \beta_\regret, \beta_\ratio,\Phi$ is feasible:

\begin{equation}
\begin{aligned}
   & \sum_{i \in \mathcal{I}} \alpha_{\revenue,i}(p) m_i + \sum_{j \in \mathcal{J}} \beta_{\revenue,j}(p) q_j \leq -\theta_\revenue  \quad \forall p \in \jaedit{\gridval} \\
   &  \sum_{i \in \mathcal{I}} \alpha_{\regret,i}(p) m_i + \sum_{j \in \mathcal{J}} \beta_{\regret,j}(p) q_j \leq \theta_\regret  \quad \forall p \in \jaedit{\gridval} \\ 
   &  \sum_{i \in \mathcal{I}} \alpha_{\ratio,i}(p) m_i + \sum_{j \in \mathcal{J}} \beta_{\ratio,j}(p) q_j \leq 0  \quad \forall p \in \jaedit{\gridval} \\
  &  - \int_{s \leq v} s d\Phi(s) - \sum_{i \in \mathcal{I} } \alpha_{\revenue,i}(p) v^i - \sum_{j \in \mathcal{J}} \beta_{\revenue,j}(p) \1(v \leq r_j) \leq 0 \quad \forall v, p \in \jaedit{\gridval} \\
   & p \1(v \ge p) - \int_{s \leq v} s d\Phi(s) - \sum_{i \in \mathcal{I} } \alpha_{\regret,i}(p) v^i - \sum_{j \in \mathcal{J}} \beta_{\regret,j}(p) \1(v \leq r_j) \leq 0 \quad \forall v, p \in \jaedit{\gridval}  \\
   & \theta_\ratio p \1(v \ge p) - \int_{s \leq v} s d\Phi(s) - \sum_{i \in \mathcal{I} } \alpha_{\ratio,i}(p) v^i - \sum_{j \in \mathcal{J}} \beta_{\ratio,j}(p) \1(v \leq r_j) \leq 0 \quad \forall v, p \in \jaedit{\gridval}  \\
   & \Phi \textnormal{ is a CDF,}
\end{aligned}
\label{eqn:lp-feasibility} \tag{$\mathcal{LP}$-\textnormal{multi}} %\tag{Feasibility-LP}
\end{equation}
and if the above program is feasible,  a solution $\Phi$  achieves $( \theta_\revenue, \theta_\regret,\theta_\ratio)$.
\end{proposition}

We can use \eqref{eqn:lp-feasibility} to compute the multi-criteria robust approximation factor $c^*(\mathcal{F})$ of a given uncertainty set $\mathcal{F}$, using the representation \eqref{eqn:relperf-decoupled} (cf. Proposition \ref{prop:relperf-equivalent}), as follows. First, using $\mathcal{F}$ and \eqref{eqn:lp-minimax}, we compute the maximin revenue $\theta^*_\revenue$, the minimax regret $\theta^*_\regret$,  and maximin ratio $\theta^*_\ratio$. By the definition of $c^*(\mathcal{F})$, for each value $c \in [0,1]$ the tuple $(\theta_\revenue, \theta_\regret,\theta_\ratio) = (c \cdot \theta^*_\revenue,\theta^*_\regret/c, c \cdot \theta^*_\ratio)$ is feasible if and only if $c \leq c^*(\mathcal{F})$. For each $c$, we can check the feasibility of $c$ with (\ref{eqn:lp-feasibility}), so we can compute $c^*(\mathcal{F})$, the largest possible feasible value of $c$, up to any desired precision by binary search on $c \in [0,1]$.  If $\gridval$ is a finite set of size $K$, then \eqref{eqn:lp-feasibility} is a linear feasibility program with $\Theta(K)$ variables and $\Theta(K^2)$ constraints.

\section{Performance of Focal Mechanisms Across Robustness Criteria} \label{sec:across}

We now evaluate the relative performance of each of the mechanisms $\mechregret, \mechrevenue, \mechratio$ against different criteria. For support $[a,b]$, we set $\gridval$ to be the values $a + i (b-a)/K$ for $i \in [K]$, with $K = 100$. To initially make the insights clear and compact, we will be interested in not just a specific uncertainty set, but a class of uncertainty sets.  We normalize the (known) support upper bound \jaedit{$b$} to 1. We denote each class of uncertainty sets as follows. Let
\begin{align*}
    \calFmu &\textnormal{ be the set of distributions with support $[0,1]$ and mean $\mu$} \\
    \calFmusigma &\textnormal{ be the set of distributions with support $[0,1]$, mean $\mu$, and variance $\sigma^2$} \\
    \calFmedian &\textnormal{ be the set of distributions with support $[0,1]$ and median $\nu$}. \\
    \calFa &\textnormal{ be the set of distributions with support $[a,1]$}.
\end{align*}

These four classes of uncertainty sets correspond to  cases that have been studied in previous papers, as discussed in \S\ref{sec:lit}.

Take $\calFmu$ as an example. For each $\mu$, we can evaluate the upper bound of how any mechanism optimized against one criterion, say regret $\mechregret = \mechregret(\calFmu)$, performs against another criterion, say revenue $\relperf(\mechregret(\calFmu),\revenue,\calFmu)$. This quantity is a function of $\mu$. \jaedit{By using \eqref{eqn:lp-across}, we can compute the best performance \jaedit{(i.e. lowest possible $\lambda$-regret)} for any such optimal $\mechregret$.} If we want to make a statement specific to an uncertainty set, this quantity is sufficient. However, in order to make a \textit{general} statement about  how a regret-optimal mechanism performs on other criteria with mean information, we will evaluate performance across instances (values of $\mu$). Let $\gridparam = \jaedit{\{0.1, 0.2, \dots, 0.9\}}$. We will evaluate the worst case performance when the unknown parameter belongs to the grid $\gridparam$,  $\min_{\mu \in \gridparam} \relperf(\mechrevenue(\calFmu),\regret,\calFmu)$.

While it is clear that a mechanism that is optimal with respect to one criterion will not necessarily be optimal with respect to a different robustness criterion, \emph{what is unclear is how suboptimal} will the initial mechanism be for an alternative criterion. The next set of results quantifies how the three focal mechanisms perform when evaluated under alternative criteria, when focusing on different uncertainty structures. We report all numerical results with two decimal points.

\begin{proposition}[Mean Information]\label{prop:across-mean}
Consider the setting with uncertainty set given by $\mathcal{F}_{\mu}$.
\begin{enumerate}%[i.)]

\item Any maximin revenue optimal mechanism $\mechrevenue(\calFmu)$ 
satisfies
\begin{align*}
\min_{\mu \in \gridparam} \relperf(\mechrevenue(\calFmu), \all, \mathcal{F}) = 0.58.
   % \inf_{\mu \in \gridparam} \relperf(\mechrevenue(\calF_\mu), \regret, \mathcal{F}_\mu) &\le 0.50\\
    %\inf_{\mu \in \gridparam} \relperf(\mechrevenue(\calF_\mu), \ratio, \mathcal{F}) &\le 0.75.
\end{align*}
\item Any minimax regret optimal mechanism $\mechregret(\calFmu)$ satisfies
\begin{align*}
\min_{\mu \in \gridparam} \relperf(\mechregret(\calFmu), \all, \mathcal{F}) = 0.44.
    %\inf_{\mu \in \gridparam} \relperf(\mechregret(\calF_\mu), \revenue, \mathcal{F}_\mu) &\le 0.19 \\
    %\inf_{\mu \in \gridparam} \relperf(\mechregret(\calF_\mu), \ratio, \mathcal{F}_\mu) &\le 0.01 .
\end{align*}

\item Any maximin ratio optimal mechanism $\mechrevenue(\calFmu)$ satisfies
\begin{align*}
\min_{\mu \in \gridparam} \relperf(\mechrevenue(\calFmu), \all, \mathcal{F}) = 0.68.
%\inf_{\mu \in \gridparam} \relperf(\mechratio(\calF_\mu), \regret, \mathcal{F}_\mu) &\le 0.93 \\
%\inf_{\mu \in \gridparam} \relperf(\mechratio(\calF_\mu), \revenue, \mathcal{F}_\mu) &\le 0.64 .
\end{align*}

\end{enumerate}    
\end{proposition}

Note that this is a negative result in the sense that we give an upper bound on the performance for \textit{any} ctiterion-specific optimal mechanism that is tailored to the value of $\mu$, using \eqref{eqn:lp-across}. We observe that, across instances,  a maximin revenue optimal mechanism can guarantee at most 58\% relative performance across robustness criteria, a minimax regret optimal mechanism can guarantee at most 44\%, and a maximin ratio optimal mechanism at most 68\%. This highlights that all mechanisms performance are fragile in this setting. 
We provide a more granular picture of  all cases in Table \ref{table:WC-Rel-robust-pricing-F-mu}, where we present all nine combinations of mechanism and criterion of evaluation.

\begin{table}[h!]
\begin{center}
\begin{tabular}{c c || c | c | c } 
 %\hline
 & & \multicolumn{3}{c}{Evaluated Criterion} \\
  \cline{3-5}
 \textbf{Mean} & & revenue & regret & ratio  \\ [0.5ex] 
 \hline
 \hline
Mechanism & revenue & 1.00  & 0.58 &  0.75 \\
 \cline{3-5}
Criterion  & regret & 0.45 & 1.00 &  0.44 \\
 \cline{3-5}
  & ratio & 0.68  &  0.93     & 1.00 \\
 \hline
\end{tabular}
\end{center}
\caption{The uncertainty set is known mean information $\calFmu$. Each cell is $\min_{\mu \in \gridparam} \relperf(\Phi^*_{\textnormal{MechanismCriterion}}(\calFmu), \textnormal{EvaluatedMetric}, \calFmu)$, the performance of a mechanism optimized for one criterion (row) when evaluated under  another criterion (column).}
\label{table:WC-Rel-robust-pricing-F-mu}
\end{table}

We observe that, under the uncertainty set $\calFmu$, the minimax regret mechanism performs poorly against the other two criteria (44\% for ratio and 45\% for regret). The maximin revenue mechanism performs reasonably under the ratio criterion (75\%), but not as well under the regret criterion (58\%). The maximin ratio mechanism performs very well under the regret criterion at 93\%, but the performance degrades at 68\% for the regret criterion.

Similar results can be obtained for the other uncertainty sets under consideration.
We present the results  for $\calFmusigma$ in Table \ref{table:WC-Rel-robust-pricing-F-mu-sigma},  for $\calFmedian$ in Table \ref{table:WC-Rel-robust-pricing-F-q}, and for $\calFa$ in Table \ref{table:WC-Rel-robust-pricing-F-a}. We observe across settings that focal mechanisms can fare quite poorly when evaluated against alternative classical robustness criteria. We also note that across cases, minimax regret mechanisms seem to perform particularly poorly against the other two criteria, compared to the maximin revenue and maximin ratio mechanisms. This is because the minimax regret rule often prescribes to put no mass on low values, a feature that adversarial Nature can take advantage of for other criteria, leading to low performance. It is also important to note that while minimax regret mechanisms might perform worse in the \textit{worst case over a large set of instances} (as shown in this section), this is not true for every instance. In instances with known high mean and fixed variance or with known high median, the minimax regret mechanism leads to a higher guarantee than the two other focal mechanisms (cf. Figure~\ref{fig:relperf-by-param-all}).

\begin{table}[h!]
\begin{center}
\begin{tabular}{c c || c | c | c } 
 %\hline
 & & \multicolumn{3}{c}{Evaluated Criterion} \\
  \cline{3-5}
 \textbf{Mean \& Variance} & & revenue & regret & ratio  \\ [0.5ex] 
 \hline
 \hline
Mechanism & revenue & 1.00  & 0.51 &  0.67 \\
 \cline{3-5}
Criterion  & regret & 0.49 & 1.00 &  0.58 \\
 \cline{3-5}
  & ratio & 0.78  &  0.71     & 1.00 \\
 \hline
\end{tabular}
\end{center}
\caption{The uncertainty set is known first two moments (mean and variance) information $\calFmusigma$. Each cell is $\min_{\mu,\sigma \in \gridparam} \relperf(\Phi^*_{\textnormal{MechanismCriterion}}(\calFmusigma), \textnormal{EvaluatedMetric}, \calFmusigma)$, the performance of a mechanism optimized for one criterion (row) when evaluated under  another criterion (column).}
\label{table:WC-Rel-robust-pricing-F-mu-sigma}
\end{table}

\begin{table}[h!]
\begin{center}
\begin{tabular}{c c || c | c | c } 
 %\hline
 & & \multicolumn{3}{c}{Evaluated Criterion} \\
  \cline{3-5}
 \textbf{Median} & & revenue & regret & ratio  \\ [0.5ex] 
 \hline
 \hline
Mechanism & revenue & 1.00  & 0.41 &  0.34 \\
 \cline{3-5}
Criterion & regret & 0.00 & 1.00 &  0.00 \\
 \cline{3-5}
  & ratio & 0.41  &  0.52     & 1.00 \\
 \hline
\end{tabular}
\end{center}
\caption{The uncertainty set is known median information $\calFmedian$. Each cell is $\min_{\nu \in \gridparam} \relperf(\Phi^*_{\textnormal{MechanismCriterion}}(\calFmedian), \textnormal{EvaluatedMetric}, \calFmedian)$, the performance of a mechanism optimized for one criterion (row) when evaluated under  another criterion (column).}
\label{table:WC-Rel-robust-pricing-F-q}
\end{table}

\begin{table}[h!]
\begin{center}
\begin{tabular}{c c || c | c | c } 
 %\hline
 & & \multicolumn{3}{c}{Evaluated Criterion} \\
  \cline{3-5}
 \textbf{Lower Bound} & & revenue & regret & ratio  \\ [0.5ex] 
 \hline
 \hline
Mechanism & revenue & 1.00  & 0.41 &  0.33 \\
 \cline{3-5}
Criterion  & regret & 0.00 & 1.00 &  0.00 \\
 \cline{3-5}
  & ratio & 0.31  &  0.53     & 1.00 \\
 \hline
\end{tabular}
\end{center}
\caption{The uncertainty set is known lower bound information $\calFa$. Each cell is $\min_{a \in \gridparam} \relperf(\Phi^*_{\textnormal{MechanismCriterion}}(\calFa), \textnormal{EvaluatedMetric}, \calFa)$, the performance of a mechanism optimized for one criterion (row) when evaluated under  another criterion (column).}
\label{table:WC-Rel-robust-pricing-F-a}
\end{table}

\section{The Best of Many Criteria: Results and Discussions}\label{sec:best-results}

\subsection{Achievable Performance}

As seen in \S\ref{sec:across}, all three focal mechanisms are not robust across criteria. In this section, we explore the question of whether there exists a mechanisms that can perform well across the three criteria.  Recall that for a given uncertainty set of distributions $\mathcal{F}$, we denote by $c^*(\mathcal{F})$, defined in \eqref{eqn:c_star}, the highest fraction of the three criteria achievable across all mechanisms.

\begin{theorem}[Robust Pricing: Best of Many Criteria]\label{thm:factor-c-lb} 
For each uncertainty set $\calF$ below, we have the following (uniform across instances) lower bound on $c^*(\calF)$.
\begin{enumerate}%[$i.)$] 
\item (Mean Information) For any $\mu \in \gridparam$, $ c^*(\calFmu) \ge 0.92$.

\item (Mean and Variance Information) For any  $\mu,  \sigma \in \gridparam$ such that $\calFmusigma \neq \emptyset$, $c^*( \calFmusigma) \ge 0.86$.

\item (Median Information) For any  $\nu \in \gridparam$, $c^*(\calFmedian) \ge 0.61$.

\item (Lower Bound Information)  For any  $a \in \gridparam$, $ c^*(\calFa) \ge 0.58$.

\end{enumerate}

\end{theorem}

In other words, for any known mean in $\gridparam$, there exists a mechanism that can ensure at least 92\% of the best performance achievable for any of the three robust criteria. For any known mean and variance, we get an 86\% guarantee; for median, 61\%, and for lower bound, 58\%. Note that $c^*(\calF) = \relperf(\mechall, \all, \calF)$. It is important to put these numbers in perspective by comparing with the performance $\relperf(\Phi, \all, \calF)$ of focal mechanisms $\Phi \in \{\mechrevenue, \mechregret, \mechratio\}$ from Section~\ref{sec:across}. 
\begin{itemize}
    \item With mean information, the best focal mechanism is $\mechratio$ with relative performance of 68\% compared with the best of 92\% ($\mechrevenue$ and $\mechregret$ are worse at $58\%$ and $44\%$).
    \item With mean and variance information, the best focal mechanism is $\mechratio$ with relative performance of 71\% compared with the best of 86\% ($\mechrevenue$ and $\mechregret$ are worse at $51\%$ and $49\%$).
    \item With median information, the best one can obtain, at 61\%, is quite large compared to the best focal mechanism $\mechratio$ at 41\% ($\mechrevenue$ and $\mechregret$ are worse at 34\% and 0\%.)
    \item With lower bound information, the best one can obtain, at 58\% is quite large compared to the best focal mechanism $\mechrevenue$ at 33\% ($\mechratio$ and $\mechregret$ are worse at 31\% and 0\%.)
    
\end{itemize}

\paragraph{Remark.}     It might seem surprising that the guarantee for mean and variance (86\%) is lower than that for mean alone (92\%) even though the former has strictly more information. In fact, it is possible for some $\mu,\sigma$ to have $c^*(\calFmusigma) < c^*(\calFmu)$ because when we add more information and $\calF$ becomes smaller, there are two effects, as seen in the definition \eqref{eqn:relperf-all-def}: (i) we minimize over a smaller uncertainty set $F \in \calF$, which makes the relative performance \textit{larger} (ii) the \textit{benchmarks} that we compare against become more competitive, i.e. $\theta^*_\revenue(\calF)$ and $\theta^*_\ratio(\calF)$ increase and $\theta^*_\regret(\calF)$ decreases, which makes the relative performance \textit{smaller}. Of course, in the extreme case when there is no uncertainty and $\calF$ is a singleton, $c^*(\calF)$ takes the maximum value of 1, but in between, either (i) or (ii) can dominate. 

\subsection{Mechanisms and Performance Across Instances}

\paragraph{Performance across instances.} While the result above highlights that one may improve the worst case factor $c^*(\calF)$, across a large number of instances $\calF$, next we explore the achievable across instances for the four classes of uncertainty sets we analyze. For $\calFmu$ we plot, as a function of $\mu$, the relative performance $\relperf(\Phi, \all, \calFmu)$ for the best mechanism $\Phi = \mechall$ as well as the focal mechanisms $\Phi \in \{\mechrevenue, \mechregret, \mechratio\}$. 
We give similar plots for the other 3 classes of uncertainty sets, showing the relative performance of each mechanism as a function of the key parameter of that class: $\calFmusigma$ as a function of $\mu$ with fixed $\sigma = 0.20$ (so the plot is two-dimensional), $\calFmedian$ as a function of $\nu$, and $\calFa$ as a function of $a$. All plots show the lower bound on the performance for $\mechall$ and the upper bounds on the performances of all other mechanisms, to emphasize the substantial provable gap between $\mechall$ and all other mechanisms.  Figure~\ref{fig:relperf-by-param-all} presents these results.
\begin{figure}[h!]
\captionsetup{justification=centering}
	\begin{subfigure}{0.48\linewidth}
    	\centering
    	\includegraphics[height=0.75\linewidth]{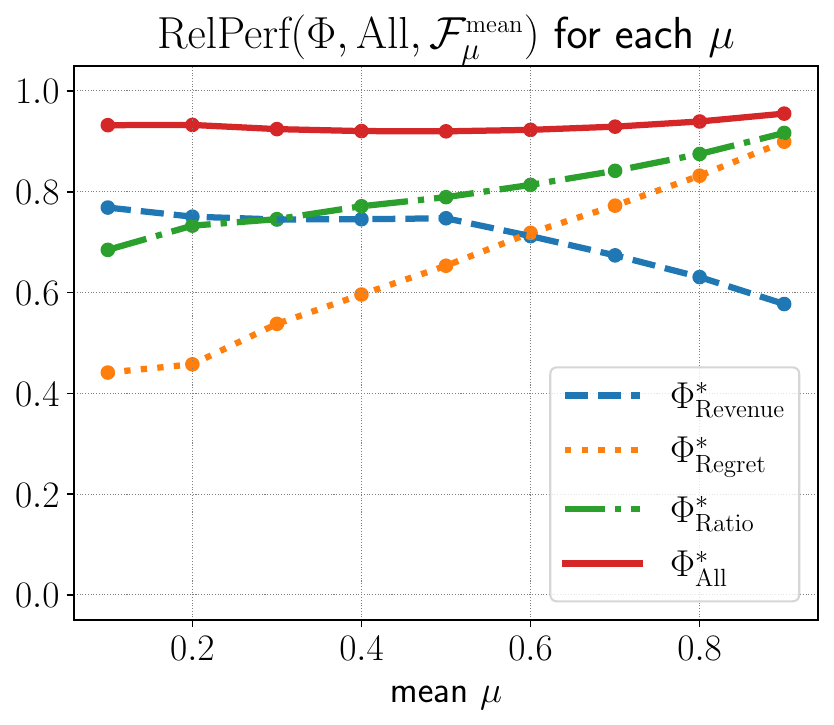}

    \end{subfigure}%
	\begin{subfigure}{0.48\linewidth}

		\centering
		\includegraphics[height=0.75\linewidth]{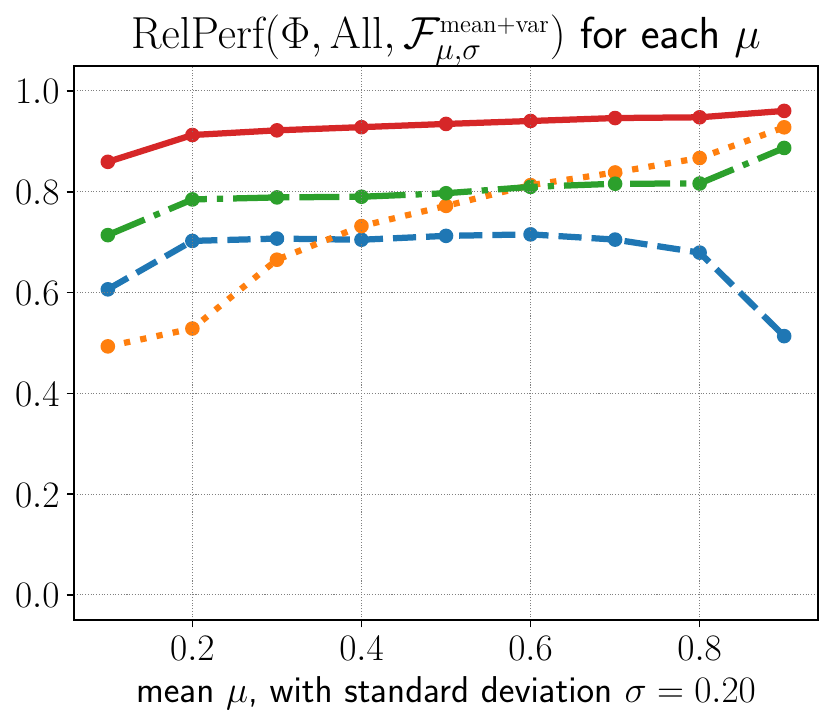}
	\end{subfigure}

	\begin{subfigure}{0.48\linewidth}
		%\vspace*{-0.5cm} % This is so it aligns nicely
		\centering
		\includegraphics[height=0.75\linewidth]{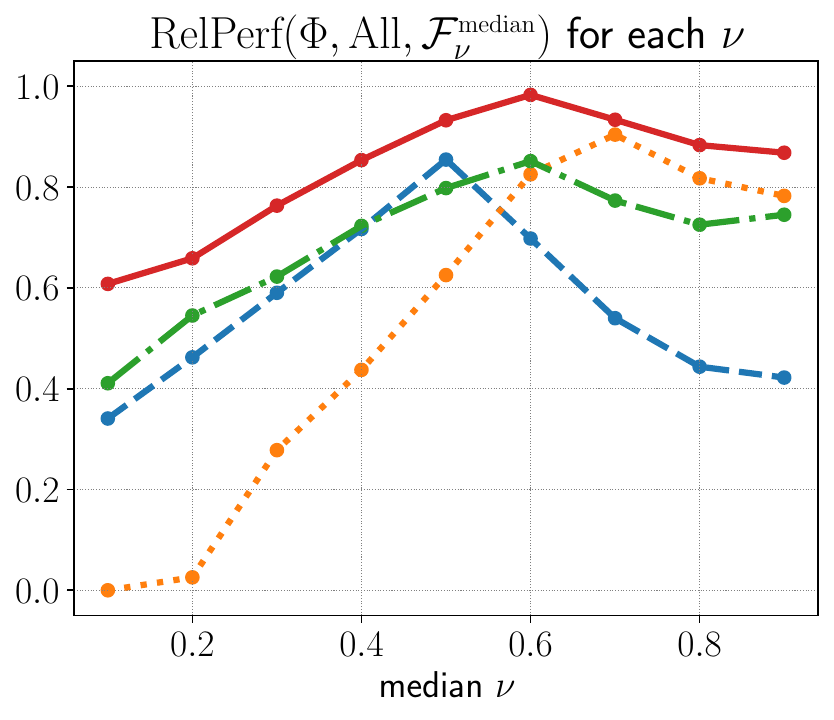}
	\end{subfigure}
 \begin{subfigure}{0.48\linewidth}
    	\centering
    	\includegraphics[height=0.75\linewidth]{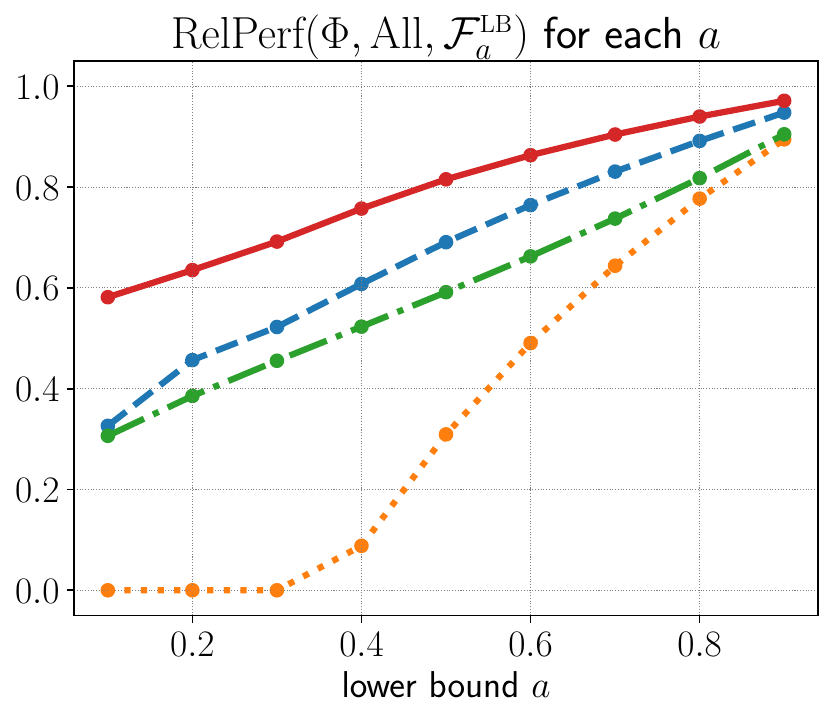}

    \end{subfigure}%
	\caption{Relative performance of different mechanisms $\Phi \in \{\mechregret, \mechrevenue, \mechratio, \mechall\}$ over all criteria as a function of the key parameter, in four classes of uncertainty sets: known mean $\calFmu$, known mean and variance $\calFmusigma$,  known median $\calFmedian$, and known lower bound $\calFa$.  }
\label{fig:relperf-by-param-all}
\end{figure}

We observe that across all instances, there is significant room to improve relative performance across criteria, when deviating from the three focal mechanisms. A striking case is that of median equal to 0.6. Despite the fact that each focal mechanism has relative performance below 85\% across criteria, the uniformly optimal mechanism achieves 98\% of performance across criteria. In general, the uniformly optimal mechanism performs around 10-15\% better than the \textit{best} focal mechanism \textit{at that parameter value}. However, if we stick to a single focal mechanism, there are parameter regions where that focal mechanism performs worse than that. (This is particularly true for $\mechrevenue$ and $\mechregret$; while $\mechratio$ has a significant performance gap against $\mechall$, the gap is fairly constant across parameter values.) The exception to this trend is the case with known lower bound $\calFa$, where $\mechrevenue$ dominates the other two focal mechanisms uniformly. This is because there are no other constraints on the distribution other than the support, so Nature can always put all its mass at the lower bound $a$ in the worst-case revenue criterion, so the revenue criterion is particularly sensitive to the mass at $a$ in particular, which is not taken into account in the other two focal mechanisms. Even in this case, however, we see that $\mechrevenue$ might not perform as well against the other two criteria, and the uniformly robust mechanism $\mechall$ significantly improves performance.

\paragraph{Form of the mechanisms.} Lastly, we illustrate the various prescriptions that emerge. To that end, we select example uncertainty sets from each of the four classes $\calFmu, \calFmusigma, \calFmedian, \calFa$ and plot the price CDF distributions of mechanisms $\Phi \in \{\mechrevenue, \mechregret, \mechratio, \mechall\}$ for comparison in Figure~\ref{fig:phi-compare}.  
\begin{figure}[h!]
\captionsetup{justification=centering}
	\begin{subfigure}{0.48\linewidth}
    	\centering
    	\includegraphics[height=0.75\linewidth]{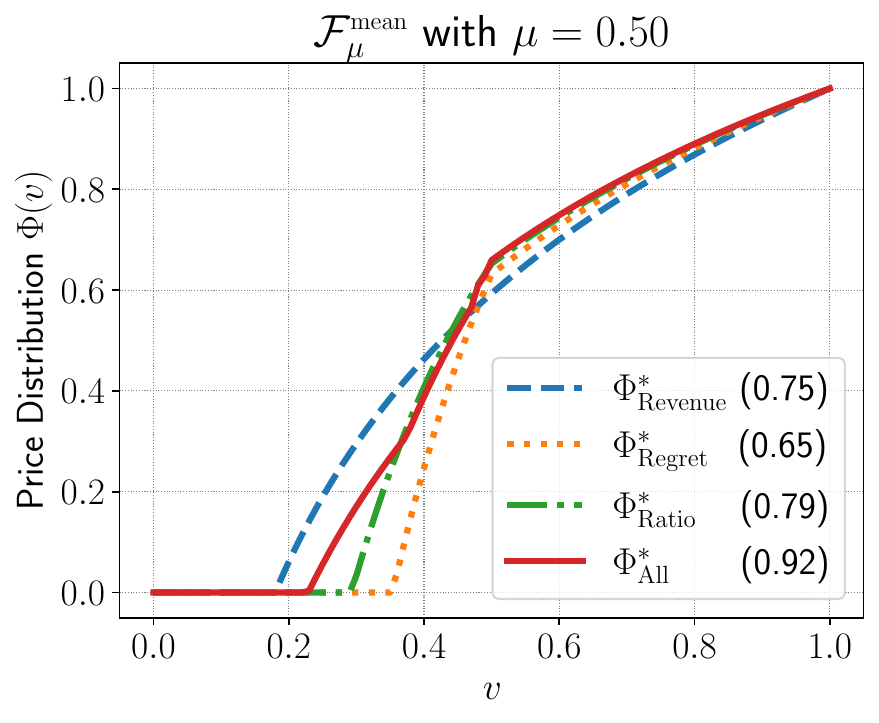}
    	\caption{mean information $\calFmu$}
    \end{subfigure}%
	\begin{subfigure}{0.48\linewidth}
		\centering
		\includegraphics[height=0.75\linewidth]{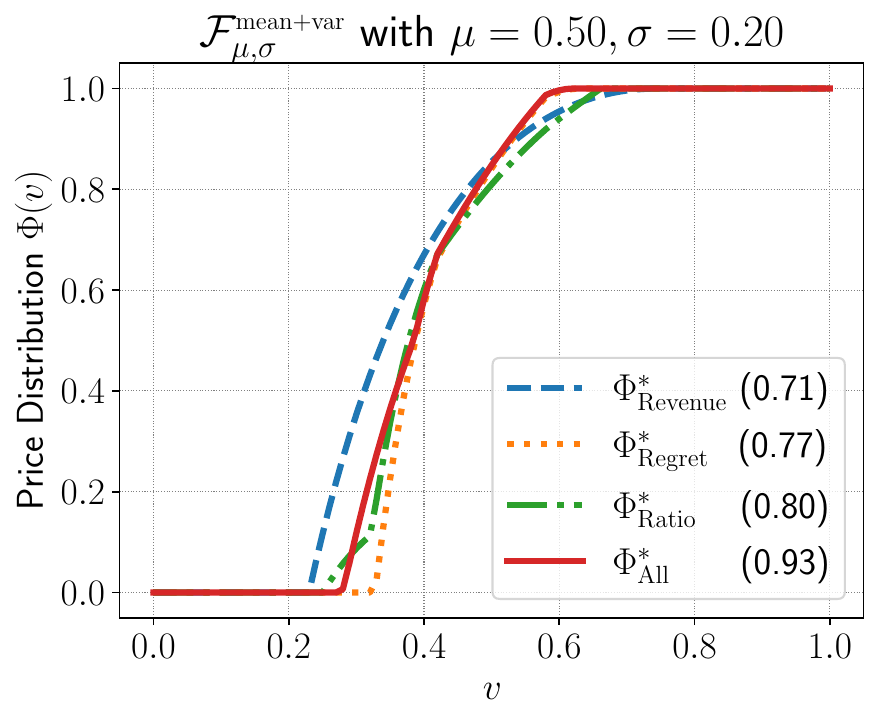}
		\caption{mean and variance information $\calFmusigma$}

	\end{subfigure}
	\begin{subfigure}{0.48\linewidth}

		\centering
		\includegraphics[height=0.75\linewidth]{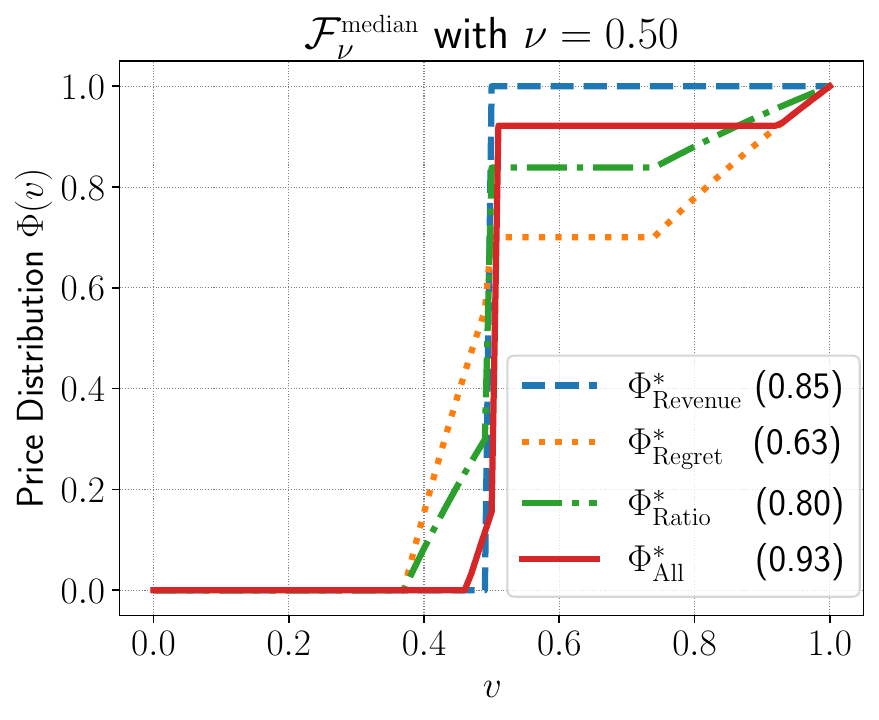}
		
		\caption{median information $\calFmedian$ }

	\end{subfigure}
	\begin{subfigure}{0.48\linewidth}
    	\centering
    	\includegraphics[height=0.75\linewidth]{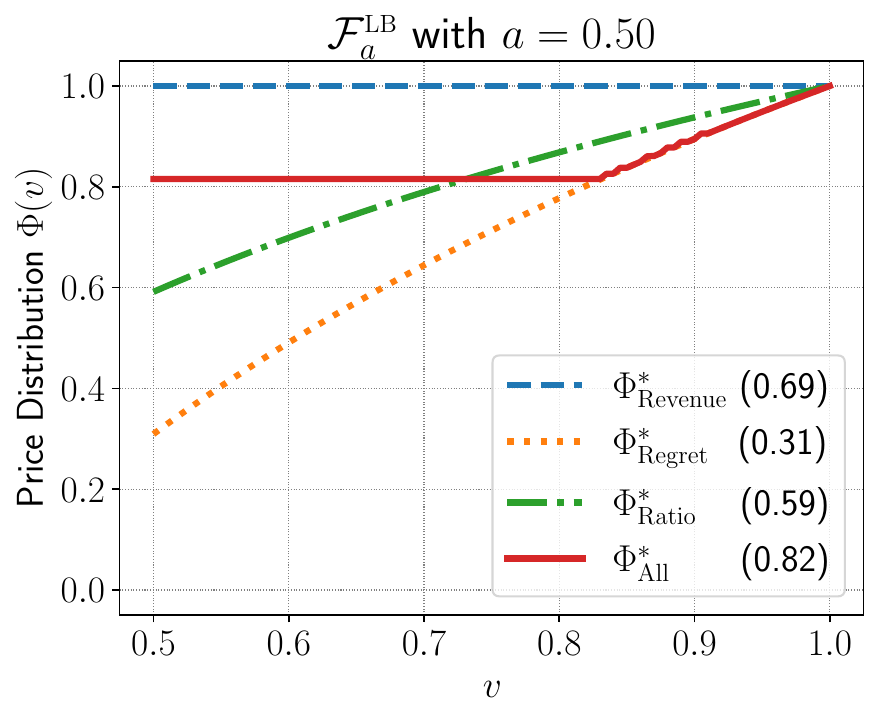}
    	\caption{lower bound information $\calFa$}
    	%\label{fig:heatmap-3-binary}
    \end{subfigure}%

	\caption{Comparison of different price distributions $\Phi \in \{\mechrevenue, \mechregret, \mechratio, \mechall\}$ for prototypical uncertainty sets in each of the four classes. Each number in parenthesis is $\relperf(\Phi,\all,\calF)$ for the corresponding $\Phi$ and $\calF$. 
 }
\label{fig:phi-compare}
\end{figure}

Each number in parenthesis of the plot legend in Figure~\ref{fig:phi-compare} is $\relperf(\Phi,\all,\calF)$, the relative performance of the corresponding mechanism $\Phi$ under uncertainty set $\calF$. As the relative performance numbers suggest, there is a lot of room for improvement. In the cases of $\calFmu, \calFmusigma, \calFmedian$, by changing from $\mechratio$ to $\mechall$, we get a 13\%, 13\%, and 8\% improvement, respectively. This highlights an interesting property in robust pricing. The mechanism $\mechall$ is not a drastic adjustment to the other mechanisms and yet can allow to obtain much stronger criteria robustness.

\section{Concluding Remarks}

In this study, we investigate robust decision-making and the extent of criteria-overfitting that may take place.  In the context of robust pricing, we derive a comprehensive analysis that enables us to answer these central questions.  We show across a variety of settings that overfitting is a real issue. Mechanisms tend to be highly tailored to the robust criterion selected and can perform very poorly when evaluated against alternative robustness criteria. Furthermore, we develop a formulation to obtain the best of many criteria.  We show that it is possible to significantly outperform any of the focal mechanisms, and significantly increase the uniform guarantees across criteria. 

 This opens up various avenues for future research. While we have focused on the classical pricing problem, a fruitful avenue may be the systematic study of classical problem classes, and more broadly delineating which problems can be fragile to the selection of a particular robust criterion, and when can such fragility be addressed.

%\newpage

\bibliographystyle{plainnat}
\bibliography{references}
%\bibliography{bibfile}

\newpage
\appendix
\pagenumbering{arabic}
\renewcommand{\thepage}{App-\arabic{page}}
\renewcommand{\theequation}{\thesection-\arabic{equation}}
\renewcommand{\thelemma}{\thesection-\arabic{lemma}}
\renewcommand{\theproposition}{\thesection-\arabic{proposition}}
\setcounter{page}{1}
\setcounter{section}{0}
\setcounter{proposition}{0}
\setcounter{lemma}{0}
\setcounter{equation}{0}

\setcounter{footnote}{0}

\begin{center}
 {\Large \textbf{Electronic Companion: 
\\ Robust Auction Design with Support Information \\}
\medskip
\ifx\blind\undefined
Jerry Anunrojwong\footnote{Columbia University, Graduate School of Business. Email: {\tt janunrojwong25@gsb.columbia.edu}}, ~
Santiago R. Balseiro\footnote{Columbia University, Graduate School of Business. Email: {\tt srb2155@columbia.edu}.}, ~ and Omar Besbes\footnote{Columbia University, Graduate School of Business. Email: {\tt ob2105@columbia.edu}.}.
\fi}
\end{center}

\vspace{-4em}

\addcontentsline{toc}{section}{Appendix} % Add the appendix text to the document TOC
\part{Appendix} % Start the appendix part
\setstretch{1.0}
\parttoc % Insert the appendix TOC
\newpage

\section{Proof: Equivalence of Two Definitions of $\relperf(\Phi,\all,\calF)$}

We observe that the relative performance over all objectives can equivalently be computed as a minimum of individual relative performances across each objective.

\begin{proposition}[Equivalent Definition of Relative Performance Over All Objectives]\label{prop:relperf-equivalent}
    \begin{align}%\label{eqn:relperf-decoupled}
        \relperf(\Phi,\all,\calF) %=  \min\Big\{\relperf(\Phi, \revenue, \mathcal{F}),      \relperf(\Phi, \regret, \mathcal{F})  , \relperf(\Phi, \ratio, \mathcal{F})   \Big\}.
    = \min\left\{ \frac{\min_{F \in \calF} \revenue(\Phi,F) }{ \theta^*_\revenue}, \frac{\theta^*_\regret}{ \max_{F \in \calF} \regret(\Phi,F) } , \frac{ \min_{F \in \calF} \ratio(\Phi,F) }{ \theta^*_\ratio }  \right\}.
    \end{align}
\end{proposition}

The difference between \eqref{eqn:relperf-all-def} and \eqref{eqn:relperf-decoupled} is that in \eqref{eqn:relperf-all-def}, the \textit{same} distribution $F$ is used to evaluate the relative performance for all 3 objectives, whereas in \eqref{eqn:relperf-decoupled}, the relative performance, and thus the worst case evaluation, of each objective is computed separately.

\begin{proof}[Proof of Proposition~\ref{prop:relperf-equivalent}]
Because, for every $F \in \calF$,
\begin{align*}
    \min_{F_1 \in \calF} \revenue(\Phi,F_1) &\leq \revenue(\Phi,F) \\
    \max_{F_2 \in \calF} \regret(\Phi,F_2) &\geq \regret(\Phi,F) \\
    \min_{F_3 \in \calF} \ratio(\Phi,F_3) &\leq \ratio(\Phi,F) ,
\end{align*}
we have, for every $F \in \calF$,
\begin{align*}
    \eqref{eqn:relperf-decoupled} \leq \min\left\{ \frac{\revenue(\Phi,F) }{ \theta^*_\revenue}, \frac{\theta^*_\regret}{ \regret(\Phi,F) } , \frac{  \ratio(\Phi,F) }{ \theta^*_\ratio }  \right\} .
\end{align*}
Therefore, $\eqref{eqn:relperf-decoupled} \leq \eqref{eqn:relperf-all-def}$. Now we will prove the converse, $\eqref{eqn:relperf-all-def} \leq \eqref{eqn:relperf-decoupled}$. Let $\eqref{eqn:relperf-decoupled} = c$. Because $\eqref{eqn:relperf-decoupled}$ is a minimum of 3 values, at least one of those values must be $c$. We first consider the case that the first term is $c$. Let $F_1^* \in \arg\min_{F_1 \in \calF} \revenue(\Phi,F_1)$. Then,
\begin{align*}
    \eqref{eqn:relperf-decoupled} &\leq \min\left\{ \frac{\revenue(\Phi,F_1^*)}{\theta^*_\revenue} , \frac{\theta^*_\regret}{\regret(\Phi,F_1^*)}, \frac{\ratio(\Phi,F_1^*)}{\theta^*_\ratio}\right\} \\
    &= \min\left\{ c , \frac{\theta^*_\regret}{\regret(\Phi,F_1^*)}, \frac{\ratio(\Phi,F_1^*)}{\theta^*_\ratio}\right\} \leq c .
\end{align*}

The other two cases (that the regret term is equal to $c$, or the ratio term is equal to $c$) allow us to proceed similarly that $\eqref{eqn:relperf-decoupled} \leq c$. We therefore conclude that $\eqref{eqn:relperf-all-def} \leq \eqref{eqn:relperf-decoupled}$.

\end{proof}

\section{Proof: Reduction to Linear Programs}\label{app:sec:lp-formulation}

\begin{proof}[Proof of Proposition~\ref{prop:minimax-lmbd-regret-lp}]

The $\lambda$-regret of a mechanism $\Phi$ can be rewritten as
\begin{align*}
    \min_{\theta} & \: \theta \\
    \text{ s.t. }  & \: \theta \geq \max_{F \in \mathcal{F}} \left\{ \lambda \left\{\max_{p\in \gridval}  p \bar{F}(p)\right\} - \int \int_{s \leq v} s d\Phi(s) dF(v) \right\},
\end{align*}
which is equivalent to
\begin{align*}
    \min_{\theta} & \: \theta \\
    \text{ s.t. }  & \: \theta \geq \max_{F \in \mathcal{F}} \left\{ \lambda p \bar{F}(p) - \int \int_{s \leq v} s d\Phi(s) dF(v) \right\} \quad \forall p \in \gridval.
\end{align*}

Recall that the set $\mathcal{F}$ defined in \eqref{eqn:uncertainty-set} has linear constraints in terms of moments and quantiles. Namely, we know that $F$ has the $i$'th moment equal to $m_i$, and it has $F(r_j) = q_j$. 
 Then the maximization in the constraint corresponding to price $p$ (which we will call the $p$-program) is given by
\begin{align*}
    \max_{dF \geq 0} & \: \int_{\gridval} \left\{ \lambda p \1(v \geq p) - \int_{s \leq v} s d\Phi(s) \right\}  dF(v)  \\
    \text{ s.t. } &\: \int_{\gridval} v^i dF(v) = m_i \quad \forall i \in \mathcal{I} \\
    & \: \int_{\gridval} \1(v \geq r_j) dF(v) = q_j \quad \forall j \in \mathcal{J}.
\end{align*}
Note here that we have $m_0 = 1$, imposing that F is a CDF. 

We now apply LP duality and analyze the dual problem instead. Let the dual variables of the $p$-program be $\alpha_{i}(p) \in \mathbb{R}$ and $\beta_{j}(p) \in \mathbb{R}$, so it becomes
\begin{align*}
    \min_{\alpha(p), \beta(p) }  \max_{dF \geq 0} & \: \int_{\gridval} \left\{ \lambda p \1(v \geq p) - \int_{s \leq v} s d\Phi(s) \right\} dF(v) \\
    & \: + \sum_{i \in \mathcal{I}} \alpha_{i}(p) \left( m_i - \int v^i dF(v) \right) + \sum_{j \in \mathcal{J}} \beta_{j}(p) \left( q_j - \int \1(v \geq r_j) dF(v) \right).
 \end{align*}
Note that the objective above may be rewritten as
\begin{align*}
\sum_{i \in \mathcal{I}} \alpha_{i}(p) m_i + \sum_{j \in \mathcal{J}} \beta_{j}(p) q_j + \int_a^b \Bigg( \lambda p \1(v \geq p) - \int_{s \leq v} s d\Phi(s)  - \sum_{i \in \mathcal{I} } \alpha_{i}(p) v^i - \sum_{j \in \mathcal{J}} \beta_{j}(p) \1(v \geq r_j) \Bigg) dF(v).
\end{align*}
The dual variables must be such that the coefficient of each $dF(v)$ is non-positive in order for the maximum to not be $\infty$. Therefore the dual problem reduces to
\begin{align*}
    \min_{\alpha(p), \beta(p) }  & \: \sum_{i \in \mathcal{I}} \alpha_{i}(p) m_i + \sum_{j \in \mathcal{J}} \beta_{j}(p) q_j \\
   \text{ s.t. } & \:  \lambda p \1(v \geq p) - \int_{s \leq v} s d\Phi(s)  - \sum_{i \in \mathcal{I} } \alpha_{i}(p) v^i - \sum_{j \in \mathcal{J}} \beta_{j}(p) \1(v \geq r_j) \leq 0 \quad \forall v \in \gridval.
\end{align*}

Returning to the problem of computing the worst case $\lambda$-regret of a mechanism $\Phi$, one may rewrite this problem as
\begin{align*}
    \min_{\theta, \alpha, \beta } & \: \theta \\
    \text{ s.t. }  & \: \theta \geq  \sum_{i \in \mathcal{I}} \alpha_{i}(p) m_i + \sum_{j \in \mathcal{J}} \beta_{j}(p) q_j  \quad \forall p \in \gridval \\
    & \: \lambda p \1(v \geq p) - \int_{s \leq v} s d\Phi(s)  - \sum_{i \in \mathcal{I} } \alpha_{i}(p) v^i - \sum_{j \in \mathcal{J}} \beta_{j}(p) \1(v \geq r_j) \leq 0 \quad \forall v, p \in \gridval.
\end{align*}

Similarly, the problem of computing minimax $\lambda$-regret is the problem of finding $\Phi$ such that the worst-case $\lambda$-regret is minimized. Given that the problem of worst-case $\lambda$-regret is already written as a minimization problem above, the minimax $\lambda$-regret is equal to the value of the above linear program where $\Phi \in \mathcal{M}$ is also a variable to optimize over:
\begin{align*}
    \min_{\Phi, \theta, \alpha, \beta } & \: \theta  \\
    & \: \theta \geq  \sum_{i \in \mathcal{I}} \alpha_{i}(p) m_i + \sum_{j \in \mathcal{J}} \beta_{j}(p) q_j  \quad \forall p \in \gridval \\
   & \:  \lambda p \1(v \geq p) - \int_{s \leq v} s d\Phi(s)  - \sum_{i \in \mathcal{I} } \alpha_{i}(p) v^i - \sum_{j \in \mathcal{J}} \beta_{j}(p) \1(v \geq r_j) \leq 0 \quad \forall v, p \in \gridval.
\end{align*}
This completes the proof. 
\end{proof}

\section{Linear Programs for Maximin Ratio}\label{sec:maximin-ratio-lp}

Throughout the main text, we characterize the maximin ratio via minimax $\lambda$-regret, computing the value for different values of $\lambda$ and use line search to find the largest $\lambda$ such that minimax $\lambda$-regret is positive. This formulation has the benefit that the $\lambda$-regret unifies all three focal objectives in one framework, but requires us to do more computation, having to solve many LPs instead of one. In this Appendix, we show that the maximin ratio can be computed with a single LP. The derivation closely parallels that of \cite[Theorem 1]{shixin-ratio} and is thus omitted. (The only difference is that we also have quantile constraints with dual variables $\beta_j(p)$, but this does not add additional complexity because they are also linear constraints that can be treated the same way as the original linear moment constraints.)

\begin{proposition}[\cite{shixin-ratio}]\label{prop:maximin-ratio-lp}
    Let $\mathcal{M}$ be a given class of mechanisms. The maximin ratio problem
    \begin{align*}
        \max_{\Phi \in \mathcal{M}} \min_{F \in \calF} \ratio(\Phi,F)
    \end{align*}
    can be equivalently written as the objective to the following program
\begin{equation}
\begin{aligned}
\max_{r,\Phi, \alpha_i(p), \beta_j(p),\gamma} & \:  r \\ \textnormal{ s.t }  & \:  
rp \leq  \gamma(p) \quad \forall p \in \gridval \\
 & \:  \int_{s \leq v} s d\Phi(s) + \sum_{i \in \mathcal{I}} \alpha_i(p) (m_i - v^{i}) + \sum_{j \in \mathcal{J}} \beta_j(p) ( r_j - \1(v \geq q_j) ) - \gamma(p) \1(v \geq p)\geq 0 \quad \forall p, v \in \gridval \\
 & \:  \Phi \in \mathcal{M}
\end{aligned}
\label{eqn:lp-ratio}  \tag{Ratio-LP} %\tag{Regret-LP-LB}
\end{equation}
\end{proposition}

Note that Proposition~\ref{prop:maximin-ratio-lp} above holds for any $\mathcal{M}$. When we want to compute the maximin ratio $\theta_\ratio$, we take $\mathcal{M}$ to be the set of all distributions on $\mathcal{G}$. When we want to compute the cross-performance (that is, evaluate the worst-case ratio of either the minimax regret or maximin revenue mechanisms), we take the set of mechanisms to be $\mathcal{M}_\old$

Note that the above works whether $\Phi \in \mathcal{M}$ (optimize over all CDFs) or $\Phi \in \mathcal{M}_{\old}$ (the cross-LP). With $\Phi \in \mathcal{M}_{\old}$, we just translate it with duality as done before (the dual variables associated with the original ratio programs are $\alpha_\new,\beta_\new$, whereas the dual variables associated with $\mathcal{M}_{\old}$ are $\alpha_\old,\beta_\old$):

\begin{equation}
\begin{aligned}
& \max_{r,\Phi, \alpha_{\new,i}(p), \beta_{\new,j}(p),\alpha_{\old,i}(p),\beta_{\old,j}(p),\gamma}  r \textnormal{ s.t } \\
& rp \leq  \gamma(p) \quad \forall p \in \gridval \\
&\int_{s \leq v} s d\Phi(s) + \sum_{i \in \mathcal{I}} \alpha_{\new,i}(p) (m_i - v^{i}) + \sum_{j \in \mathcal{J}} \beta_{\new,j}(p) ( r_j - \1(v \geq q_j) ) - \gamma(p) \1(v \geq p)\geq 0 \quad \forall p, v \in \gridval \\
& r_\old \geq  \sum_{i \in \mathcal{I}} \alpha_{\old,i}(p) m_i + \sum_{j \in \mathcal{J}} \beta_{\old,j}(p) q_j  \quad \forall p \in \gridval \\
& \lambda_{\old} p \1(v \geq p) - \int_{s \leq v} s d\Phi(s)  - \sum_{i \in \mathcal{I} } \alpha_{\old,i}(p) v^i - \sum_{j \in \mathcal{J}} \beta_{\old,j}(p) \1(v \geq r_j) \leq 0 \quad \forall p,v \in \gridval  \\
& \Phi \text{ is a CDF.}
\end{aligned}
\label{eqn:lp-ratio-cross}  \tag{Ratio-LP-Cross} %\tag{Regret-LP-LB}
\end{equation}

As before, if the old objective is regret, we set $\lambda_\old = 1, r_\old = \theta_\regret$ and if the old objective is revenue, we set $\lambda_\old = 0, r_\old = - \theta_\revenue$.

\end{document}